\documentclass[12pt]{article}
\usepackage{amsmath,amssymb,amsfonts,amsthm,graphicx,color}
\usepackage{bbm}
\usepackage{mathrsfs}
\usepackage{hyperref}
\usepackage{mathtools}
\usepackage{graphicx,type1cm,eso-pic,color}
\usepackage{amsmath, amssymb, amsthm}
\newtheorem{theorem}{Theorem}[section]
\newtheorem{lemma}[theorem]{Lemma}
\newtheorem{proposition}[theorem]{Proposition}
\newtheorem{remark}[theorem]{Remark}

\newtheorem{definition}[theorem]{Definition}
\newtheorem{ex}[theorem]{Example}
\newcommand{\norm}[1]{\left\|#1\right\|}

\usepackage{arxiv}

\usepackage{hyperref}       
\usepackage{url}            
\usepackage{booktabs}       
\usepackage{nicefrac}       
\usepackage{microtype}      
\usepackage{lipsum}		
\usepackage{graphicx}
\usepackage{natbib}
\usepackage{doi}
\allowdisplaybreaks
\title{Approximate Controllability of Fractional Hemivariational Inequalities in Banach Spaces}
\author{Bholanath Kumbhakar${}^{1}$, Deeksha${}^{1}$,  and Dwijendra Narain Pandey${}^{1,}\thanks{Corresponding author email: dwij@ma.iitr.ac.in}$\\	 \footnotesize{$^1$Department of Mathematics, Indian Institute of Technology Roorkee, Roorkee 247 667, India}}
\date{}
\allowdisplaybreaks

\usepackage{xcolor}
\allowdisplaybreaks
\usepackage[pagewise]{lineno}
\usepackage{graphicx,eurosym}
\usepackage{hyperref}
\usepackage{mathtools}
\colorlet{darkblue}{blue!50!black}
\hypersetup{
	colorlinks,%
	citecolor=red,%
	filecolor=red,%
	linkcolor=blue,%
	urlcolor=blue,%
	pdfnewwindow=true,%
	pdfstartview={FitH}}

\let\originalleft\left
\let\originalright\right
\renewcommand{\left}{\mathopen{}\mathclose\bgroup\originalleft}
\renewcommand{\right}{\aftergroup\egroup\originalright}

\begin{document}
	\maketitle \setcounter{page}{1} \numberwithin{equation}{section}
	\newtheorem{assumption}{Assumption}
	\newtheorem{Pro}{Proposition}[section]
	\newtheorem{Ass}{Assumption}[section]
	\newtheorem{Def}{Definition}[section]
	\newtheorem{Rem}{Remark}[section]
	\newtheorem{app}{Appendix:}
	\newtheorem{ack}{Acknowledgement:}

	\begin{abstract}
		In this paper, we derive the approximate controllability of fractional evolution hemivariational inequalities in reflexive Banach spaces involving Caputo fractional derivatives. We first show that the original problem is connected with a fractional differential inclusion problem involving the Clarke subdifferential operator, and then we prove the approximate controllability of the original problem via the approximate controllability of the connected fractional differential inclusion problem. In the meantime, we face the issue of convexity in the multivalued fixed point map due to the nonlinear nature of the duality map in reflexive Banach spaces involved in the expression of the control. We resolve this convexity issue and deduce main result. We also justify the abstract finding of this paper through an example.

	\end{abstract}

\section{Introduction}
In 2008, \cite{MR2375486} introduced and systematically studied a new kind of coupled dynamical systems on finite-dimensional spaces, called differential variational inequalities (DVIs, for short), which are formulated as a combination of (partial) differential equations and time-dependent variational inequalities. It was
shown that DVIs can serve as a powerful and useful mathematical tool to model and solve a variety of problems in engineering areas, such as dynamic vehicle routing problems, electrical circuits with ideal diodes, Coulomb frictional problems for bodies in contact, economical dynamics, dynamic traffic networks, and so on.

On the other hand, hemivariational inequalities were introduced by \cite{MR896909, MR1385670} in early 1980s to study engineering problems involving nonsmooth, nonmonotone and possibly multivalued constitutive relations and boundary conditions for deformable bodies. Since multivalued and nonmonotone constitutive laws appear often in applications, recently, \cite{MR3871422} introduced the
new notion of a differential hemivariational inequality (DHVI, for short). DHVI
is a valuable and efficient mathematical modeling tool to explore the nonsmooth
contact problems in mechanics, semipermeability problems, abnormal diffusion phenomena, etc. Hemivariational inequalities prove to be valuable in addressing mathematical challenges posed by problems involving multivalued and nonmonotone constitutive laws and boundary conditions. These multivalued relationships are derived from nonsmooth and nonconvex superpotentials, employing the generalized gradient of Clarke. In situations where convex superpotentials are at play, hemivariational inequalities simplify to variational inequalities. For a comprehensive understanding of the origins of hemivariational inequalities and their mathematical theory, readers are encouraged to consult \cite{ MR896909, MR1385670,MR1304257,MR4375652} and the associated references.

In last decades the study of evolution equations of fractional order, this is, evolution equations where the derivative of arbitrary order is used instead of a derivative of
positive integer order, has received great attention from researchers. This increasing interest is motivated both by important applications of the theory, and by considerations
of a mathematical nature. Indeed, many phenomena arising from several scientific fields
including analysis of viscoelastic materials, heat conduction in materials with memory,
electrodynamics with memory, signal processing, control theory, nonlinear dynamics and
stochastic processes, are conveniently described by fractional evolution equations. In addition, many research works in different applied sciences have demonstrated that models
involving fractional derivatives are more accurate to represent some natural phenomena
than models involving classical derivatives. We refer the reader to the monographs and
papers \cite{MR2962045}, \cite{MR747787}, \cite{MR3146657}, \cite{MR1890104} which contain many applications.

The underlying idea that motivated this article is that control theory is unquestionably one of the most interdisciplinary research areas as most modern applications involve control theory. Control theory, on the other hand, has been a discipline in which many mathematical ideas and methods have melted to produce a new body of important mathematics. As a result, it is now a rich intersection of engineering and mathematics. For a  comprehensive review of control theory and applications, one may refer \cite{MR516812} and citations given therein.

Within controllability, concepts like exact, null, and approximate controllability, the later holds particular more importance in practical scenarios. In real-world applications, achieving approximate controllability in abstract semilinear systems, especially in population dynamics, often takes precedence over exact controllability. Here we want to mention two works \cite{MR1943766,MR4577651} which contribute significantly to this area. In \cite{MR1943766} the author explores approximate controllability in a semilinear control system derived from a linear age-dependent and spatially structured population model. In \cite{MR4577651} author establishes approximate controllability results for a semilinear population model involving diffusion and a nonlocal birth process. For further insights, interested readers are encouraged to delve into the content of the book \cite{MR1797596}.

The purpose of this paper is to provide some suitable sufficient conditions for the approximate controllability of the following fractional evolution hemivariational inequality
\begin{equation}\label{HV1}
	\begin{cases}
 	\langle -{}^C D_t^\alpha q(t) +Aq(t)+Bu(t), v^*\rangle_{X,X^*} +F^0(t, q(t); H^*v^*)\ge 0, \forall v^*\in X^*, t\in I=[0,a]\\
		q(0)=x_0.
	\end{cases}
\end{equation}

 Here $\langle \cdot, \cdot \rangle_{X, X^*}$ denotes the duality pairing between the Banach spaces $X$ and its dual space $X^*$. The notation ${}^C D_t^\alpha$ stands for the Caputo fractional derivative of order $\alpha$ with $\frac{1}{2}<\alpha<1$. The linear (not necessarily bounded) operator $A: D(A)\subset X\to X$ is the infinitesimal generator of a strongly continuous semigroup $\{T(t)\}_{t\ge 0}$ on $X$. The notation $F^0(t, \cdot;\cdot)$ stands for the generalized Clarke directional derivative (see \cite[page 25]{MR1058436}) of a locally Lipschitz function $F(t,\cdot): X\to \mathbb{R}$, and $H: X^*\to X$ is a bounded linear operator. The control function $u$ takes value in $L^2(I, U)$, and the admissible controls set $U$ is a separable Hilbert space, $B$ is a bounded linear operator from $U$ into $X$.

 Nonlocal problems of Caputo type have attracted the attention of the scientific community in
recent years by its application to models with anomalous diffusion in physics \cite{MR1809268}, finance \cite{MR2263769}
and hydrology \cite{benson2000application}. This type of operators have also been studied in various mathematical contexts
as analysis of PDEs, numerical analysis, operator theory and probability, see \cite{MR1278077,MR3316530}, \cite{diethelm2002analysis} and references therein. In particular, we would like to mention the interesting work of
Baeumer, Meerschaert and Nane \cite{MR2491905}, where the authors prove the equivalence of the heat Cauchy
problems
\begin{equation}
    D^{\frac{1}{2}}u=\Delta u~\text{in}~\mathbb{R}^N\times [0,\infty),~u(x,0)=u_0(x), ~x\in \mathbb{R}^N;
\end{equation}
and 
\begin{equation}
    u_t=\Delta^2u+\frac{\Delta u_0}{\sqrt{ \pi t}}~\text{in}~Q, u(x,0)=u_0(x), ~x\in \mathbb{R}^N,
\end{equation}
for smooth initial data.

Considerable advancements have been achieved in tackling the solvability and approximate controllability problems associated with hemivariational inequalities. Notably, the works of \cite{MR3280853, MR3411720} may be marked as initial important  exploration of the approximate controllability of hemivariational inequalities with integer order in Hilbert spaces. Building upon this foundation, the authors in \cite{MR3512753} extended these findings to derive results for the approximate controllability of fractional order evolution hemivariational inequalities involving Caputo fractional derivatives, still within the framework of Hilbert spaces. Recent years have witnessed many intriguing results concerning controllability problems in Hilbert spaces, as evidenced by \cite{MR3927856, MR4128438, MR4560898, MR3937061} and related references.

Upon conducting an extensive literature review, it becomes evident that a considerable amount of research has been dedicated to addressing the approximate controllability of semilinear evolution problems linked to hemivariational inequalities, particularly when the state space is confined to Hilbert spaces. 

Despite this, the exploration of approximate controllability problems in Banach spaces described by semilinear evolution hemivariational inequalities still limited in the existing literature.

Our result extends the earlier mentioned results into the following directions:
\begin{itemize}
	\item[(i)] We broaden the scope by considering the state space $X$ as a super-reflexive Banach space (see Appendix for the definition), in contrast to the separable Hilbert spaces considered earlier. Therefore, our work is a direct generalization of the works presented in \cite{MR3280853, MR3411720}. We refer to Remark \ref{RR} in Section 3 for an explanation of why the choice of super-reflexive Banach spaces is needed.
	\item[(ii)] The paper offers a unique solution to a challenge introduced by assuming $X$ as a super-reflexive Banach space, which presents issues of convexity due to the nonlinear nature of the duality mapping arising in the expression of the control. Such issues are not present when $X$ is a separable Hilbert space. Therefore, the paper's novelty is its successful resolution of the convexity problem, paving the way for approximate controllability of the fractional hemivariational control system in which $X$ is a super-reflexive Banach space. We broadly elaborate more on this fact in Section 3.
	\item[(iii)] We prove the approximate controllability result for the problem \eqref{HV1} using Hypothesis (F3) given as follows:
	\begin{itemize}
			\item[(F3)] there exists a function $\eta\in L^\frac{1}{\alpha_1}(I, \mathbb{R}^+)$ such that
		\begin{equation}
			\norm{\partial F(t,x)}=\sup\{\norm{z}_{X^*}: z\in \partial F(t,x)\}\le \eta(t), t\in I, x\in X.
		\end{equation}
		Here $\partial F$ stands for the Clarke subdifferential (see \cite[page 27]{MR1058436}) of a locally Lipschitz function $F(t,\cdot): X\to \mathbb{R}$ and $0<\alpha_1<\alpha$.
	\end{itemize}
\end{itemize}
The paper is structured into five sections:
\begin{itemize}
	\item[(1)] Section 1 is the introduction where we discuss the motivation of the considered problem and related literature review.
	\item[(2)] Section 2 revisits preliminary concepts used in subsequent sections.
	\item[(3)] In Section 3, the approximate controllability results are provided.
	\item[(4)] Section 4 provides an application that illustrates the abstract results of the paper.
	\item[(5)] Section 5 serves as an Appendix section.
\end{itemize}

\section{Preliminaries}
This section introduces some basic definitions and notations that will be used throughout the paper. Throughout the manuscript we denote $\mathcal{L}(X,Y)$ to be the Banach space of bounded linear operators from the Banach space $X$ to the Banach space $Y$ and by $\mathcal{L}(X)$ we mean the space of all bounded linear operators on $X$. Further, in this manuscript, we denote $w-X$ to be the space $X$ furnished with weak topology. Further, integrations in Banach spaces are considered in the sense of Lebsegue-Bochner.
\subsection{Basic definitions on fractional calculus}
In this section, we state some definitions, notations and preliminary facts about fractional calculus. We start with the following definition of Riemann-Liouville fractional integral.
\begin{definition}
The Riemann-Liouville fractional integral of order $\alpha$ with $0<\alpha<1$ for a function $f\in L^1(I,X)$ is defined as  
\begin{equation}
I^{\alpha}f(t)=\frac{1}{\Gamma(\alpha)}\int^{t}_{0}\frac{f(s)}{(t-s)^{\alpha}}ds , ~t\in I,
\end{equation}
 where $\Gamma$ is the gamma function.
\end{definition}
\begin{definition}
The Caputo derivative of order $\alpha, ~0<\alpha<1$ for a function $f\in C^1(I,X)$ is given as 
\begin{equation}
^CD_t^{\alpha}f(t)= \frac{1}{\Gamma(1-\alpha)}\int^{t}_{0}\frac{f^{\prime}(s)}{(t-s)^{\alpha}}ds, ~t\in I. 
\end{equation}
\end{definition}
\subsection{Concept of Multivalued maps}
 Let us mention some basic results concerning multivalued maps. For this, we refer to the books \cite{MR1831201},\cite{MR2976197}.
\begin{definition}
A multivalued map $\Gamma:X\multimap Y$ is said to be
	\begin{itemize}
		\item  convex  valued if the set $\Gamma(x)$ is convex for every $x\in X$ and closed valued if $\Gamma(x)$ is closed for all $x\in X$.
		\item  bounded if $\Gamma$ maps bounded sets in $X$ into bounded sets in $Y$. Namely, if $B_X$ is a bounded set in $X$, then the set $\Gamma(B_X)$ is bounded in $Y$, that is, 
		\begin{equation*}
			\sup_{x\in B_X}\{\sup\{\norm{y}:y\in \Gamma(x)\}\}<\infty.
		\end{equation*}
		\item  upper semicontinuous at the point $x_0\in X$ if for each open set $O_Y\subset Y$ containing $\Gamma(x_0)$, there exists an open neighbourhood $O_X$ of $x_0$ such that $\Gamma(O_X)\subset O_Y$.
	\end{itemize}
\end{definition}
Having defined some notion of multivalued maps, we are ready to introduce multivalued measurable functions.\\
Let $\Omega\subset \mathbb{R}$ be a measurable set and $\Sigma$ be the $\sigma$- algebra of subsets of $\Omega$. Also let $X$  be a separable reflexive Banach space.
\begin{definition}
	A multivalued map $\Gamma:\Omega\multimap X$ is said to be measurable if for every $C_X\subset X$ closed, the set 
	\begin{equation*}
		\{w\in \Omega: \Gamma(w)\cap C_X\neq \phi\}\in \Sigma.
	\end{equation*}
\end{definition}
Let $\mathcal{B}(X)$ be the Borel $\sigma$- algebra of subsets of $X$.
\begin{definition}
	We say that the multivalued mapping $\Gamma: \Omega\times X\multimap X$ is $\Sigma\times \mathcal{B}(X)$ measurable if
	\begin{equation*}
		\Gamma^{-1}(C_X)=\{(w,x)\in \Omega\times X: \Gamma (w,x)\cap C_X\neq \phi\}\in \Sigma\times \mathcal{B}(X),
	\end{equation*}
	for any closed set $C_X\subset X$.
\end{definition} 
\subsection{Compactness in Banach spaces}
We now introduce the concept of uniformly integrability, which plays a crucial role in determining compactness properties of $L^p, 1\le p<\infty$ spaces.
\begin{definition}
A subset $\Xi\subset L^1(I;X)$ is called \textbf{uniformly integrable} if for every $\epsilon>0$ there is a $\delta(\epsilon)>0$ such that $\displaystyle \int_{E}\norm{g(s)}ds\le \epsilon,$
for every measurable subset $E\subset I$ whose Lebesgue measure is less than or equal to $\delta(\epsilon)$, and uniformly with respect to $g\in \Xi$.
\end{definition}
\begin{remark}\label{rk}
	Suppose $\{g_n\}_{n\in \mathbb{N}}\subset L^1(I,X)$ is a sequence of functions such that 
	$\norm{g_n(t)}\le \eta(t), ~\text{a.a.}~t\in I,$
	for some $\eta\in L^1(I,\mathbb{R}^+)$,~ $\forall n\in \mathbb{N}$ . Then $\{g_n\}_{n\in \mathbb{N}}$ is uniformly integrable.
\end{remark}
We need some weak compactness criterion in $L^1$ space. The best-known result in this direction is the celebrated Dunford-Pettis Theorem.
\begin{theorem}[Dunford Pettis Theorem]\label{Thm2.2}\cite{MR3524637}
	If $(I, \Sigma, \mu)$ is a finite measure space, $X$ is reflexive, and $K\subset L^1(I, X)$ is bounded, then $K$ is relatively weakly compact in $L^1(I, X)$ if and only if it is uniformly integrable.
\end{theorem}
We state an elementary Lemma that is useful in our subsequent analysis. 
Following Step 3, Step 4 of the proof of \cite[Theorem 4.1]{MR4076740} we can prove the following Lemma.
\begin{lemma}\label{Lem3.6}
Let the operator $\Xi: L^1(I,X)\to C(I,X)$ be such that
		\begin{equation*}
			\Xi(g)(t)=\int_{0}^{t}(t-s)^{\alpha-1}T_{\alpha}(t-s)g(s)ds,~ t\in I,
		\end{equation*}
where $\frac{1}{2}<\alpha<1$ and $\{T_{\alpha}(t)\}_{t\ge 0}$ is a family of bounded linear operators  given in \eqref{Talpha} which is compact for $t>0$. If $\{g_n\}_{n\in \mathbb{N}}\subset L^1(I,X)$ be such that
\begin{equation}
    \norm{g_n(t)}_X\le \eta(t),~\text{for a.a.}~t\in I,
\end{equation}
for some $\eta\in L^{\frac{1}{\alpha_1}}(I,\mathbb{R}^+)$ with $0<\alpha_1<\alpha$, then the sequence $\{\Xi(g_n)\}_{n\in \mathbb{N}}$ is relatively compact in $C(I,X)$.
	\end{lemma}
We consider the following fixed point theorem, which we are going to use for our results.
\begin{theorem}\label{fixed}\cite{MR46638}
	Let $X$ be a Hausdorff locally convex topological vector space, $K$ a compact convex subset of $X$, and $\Gamma: K\multimap K$ an upper semicontinuous multimap with closed, convex values. Then the multimap $\Gamma$ has a fixed point in $K$.
\end{theorem}
\subsection{Nonsmooth Analysis}
In what follows, let us define Clarke subdifferential (see \cite[page 27]{MR1058436}) of a locally Lipschitzian functional $F: X\to \mathbb{R}$. We denote by $F^0(y;z)$ the Clarke generalized directional derivative of $F$ at $y$ in the direction $z$, that is
\begin{equation}
	F^0(y;z)=\lim_{\epsilon\to 0^+}\sup_{\xi\to y}\frac{F(\xi+\epsilon z)-F(\xi)}{\epsilon}.
\end{equation} 
Recall also that the Clarke subdifferential of $F$ at $y$, denoted by $\partial F(y)$ is a subset of $X^*$ given by
\begin{equation}
	\partial F(y)=\{y^*\in X^*: F^0(y;z)\ge \langle y^*,z\rangle_{X^*,X}, \forall z\in X\}.
\end{equation}
The following basic properties of the generalized directional derivative and the generalized gradient are important in our main results.
\begin{lemma}\label{Lemma 2.14} \cite[Proposition 2.1.2]{MR1058436}
	If $F: X\to \mathbb{R}$ is a locally Lipschitz function, then
	\begin{itemize}
		\item[(i)] for every $z\in X$, one has 
		\begin{equation}
			F^0(y;z)=\max\{\langle y^*,z\rangle_{X^*,X}: ~\text{for all}~y^*\in \partial F(y)\}.
		\end{equation}
		\item[(ii)] for every $y\in X$, the gradient $\partial F(y)$ is a nonempty, convex, weak$^*$-compact subset of $X^*$ and $\norm{y^*}_{X^*}\le K$ for any $y^*\in \partial F(y)$ (where $K>0$ is the Lipschitz constant of $F$ near $y$).
		\item[(iii)] the graph of the generalized gradient $\partial F$ is closed in $X\times w-X^*$ topology, that is, if $\{y_n\}_{n\in \mathbb{N}}\subset X$ and $\{y_n^*\}_{n\in \mathbb{N}}\subset X^*$ are sequences such that $y_n^*\in \partial F(y_n)$ and $y_n\to y$ in $X$, $y_n^*\to y^*$ weakly in $X^*$, then $y^*\in \partial F(y)$. Here, $w-X^*$ denotes the Banach space $X^*$ furnished with the weak topology.
		\item[(iv)] the multifunction $y\multimap \partial F(y)\subset X^*$ is upper semicontinuous from $X$ into $w-X^*$.
	\end{itemize} 
\end{lemma}
\begin{lemma}\label{Lemma 2.15}  \cite[Proposition 3.44]{MR2976197}
	Let $X$ be a separable reflexive Banach space and $F: I\times X\to \mathbb{R}$ be a function such that $F(\cdot,x)$ is measurable for all $x\in X$ and $F(t,\cdot)$ is locally Lipschitz on $X$ for almost all $t\in I$. Then, the multifunction $(t,x)\multimap \partial F(t,x)$ is measurable, where $\partial F$ denotes the Clarke subdiffeential of $F(t,\cdot)$.
\end{lemma}
\subsection{Hypothesis}
Throughout the paper, we impose the following Hypotheses:
\begin{itemize}
	\item[(T)] The operator $A$ generates a strongly continuous semigroup $\{T(t)\}_{t\ge 0}$ of bounded linear operators $T(t): X\to X$ and $T(t)$ is compact for $t>0$. Moreover, there exists a constant $M>0$ such that
	\begin{equation}
		\norm{T(t)}_{\mathcal{L}(X)}\le M, ~~0\le t\le a.
	\end{equation}
	\item[(U)] The fractional linear control system 
	\begin{equation}\label{Linear}
		\begin{cases}
			^CD_t^{\alpha}q(t)=Aq(t)+Bu(t),~~ t\in I\\
			q(0)=x_0
		\end{cases}
	\end{equation}
	is approximately controllable in $I$.
\end{itemize}
We also assume that the function $F: I\times X\to \mathbb{R}$ satisfies the following Hypotheses.
\begin{itemize}
	\item[(F1)] the function $t\mapsto F(t,x)$ is measurable for all $x\in X$.
	\item[(F2)] the function $x\mapsto F(t,x)$ is locally Lipschitz continuous for a.e. $t\in I$.
	\item[(F3)] there exists a function $\eta\in \ L^{\frac{1}{\alpha_1}}  (I, \mathbb{R}^+)$ such that
		\begin{equation}
			\norm{\partial F(t,x)}=\sup\{\norm{z}_{X^*}: z\in \partial F(t,x)\}\le \eta(t), t\in I, x\in X.
		\end{equation}
\end{itemize}

Note that under Hypotheses (F1)-(F3), by virtue of Lemmas \ref{Lemma 2.14} and \ref{Lemma 2.15} the multimap $G:I\times X\multimap X^*$ given by $G(t,x)=\partial F(t,x)$ satisfies the following Hypotheses (G1)-(G3).
\begin{itemize}
	\item[(G1)] the multimap $G$ has nonempty, convex and weakly compact values; and the multimap $G(\cdot,x): I\multimap X^*$ is measurable for all $x\in X$,
	\item[(G2)] the multimap $G(t,\cdot): X\multimap X^*$ has strongly weakly closed graph, that is if $x_n\to x$ in $X$, $y_n\rightharpoonup y$ in $X^*$ with $y_n\in G(t,x_n)$, then $y\in G(t,x)$.
	\item[(G3)] there exists a function  $\eta\in \ L^{\frac{1}{\alpha_1}}(I, \mathbb{R}^+)$ such that
	\begin{equation}
		\norm{G(t,x)}=\sup\{\norm{y}_{X^*}: y\in G(t,x)\}\le \eta(t), ~\text{a.a.}~t\in I, x\in X.
	\end{equation} 
\end{itemize}

\begin{remark}
In view of Lebourg Mean Value Theorem \cite[Theorem 2.3.7]{MR1058436}, under Hypotheses (F2) and (G3) it is clear that the function $F(t,\cdot):X\to \mathbb{R}$ must be Lipschitz continuous. Therefore, the first approximate controllability result given by Theorem \ref{MR} holds under the assumption that the function $F(t,\cdot):X\to \mathbb{R}$ is Lipschitz continuous. 
\end{remark}

\section{Main Result}
This section contains the approximate controllability of the fractional evolution hemivariational control problem \eqref{HV1}.

We start by showing that the fractional evolution hemivariational control problem \eqref{HV1} is connected to the following fractional differential inclusion involving Clarke subdifferential
\begin{equation}\label{PC1}
	\begin{cases}
		{}^C D_t^\alpha q(t)\in Aq(t)+H\partial F(t, q(t))+Bu(t), t\in I=[0,a]\\
		q(0)=x_0.
	\end{cases}
\end{equation}
In the above, ${}^C D_t^\alpha q(t)$ denotes the Caputo fractional derivative of the function $q(t)$ of order $\alpha$ with $\frac{1}{2} < \alpha< 1$.
The linear operator $A:D(A)\subset X\to X$ is the infinitesimal generator of a strongly continuous semigroup $\{T(t)\}_{t\ge 0}$ on a separable super-reflexive Banach space $X$. The multimap $\partial F$ stands for the Clarke subdifferential of a locally Lipschitz function $F(t,\cdot): X\to \mathbb{R}$ and $H: X^*\to X$ is a bounded linear operator.

We need to introduce the solution concept of the fractional differential inclusion \eqref{PC1}. We now define the selection map $S_{\partial F}: C(I,X)\multimap L^{1}(I,X^*)$ as follows:
\begin{equation}\label{Selection}
	S_{\partial F}(q)=\{f\in L^{1}(I,X^*): f(t)\in \partial F(t, q(t))~\text{a.a.}~t\in I\}.
\end{equation}
By assuming that the multimap $S_{\partial F}$ is well defined, we can define the notion of solutions of the Problem \eqref{PC1}.
\begin{definition}\label{solution of DI}
Fix $u(\cdot)\in L^2(I,U)$. A function $q(\cdot)$ is a solution of the problem \eqref{PC1} if $q(\cdot)$ is a solution of the problem
\begin{equation}\label{NHGS}
	\begin{cases}
		{}^C D_t^\alpha q(t)=Aq(t)+Bu(t)+Hf(t),~~ ~\text{a.a.}~t\in I\\
		q(0)=x_0,
	\end{cases}
\end{equation}
for some $f\in S_{\partial F}(q)$.
\end{definition}
Suppose $q$ is a solution of the inclusion problem \eqref{PC1}. Then, according to the Definition \ref{solution of DI} we obtain
\begin{equation}\label{duality product1}
    {}^C D_t^\alpha q(t)=Aq(t)+Bu(t)+Hf(t),~~ ~\text{a.a.}~t\in I,
\end{equation}
for some $f\in S_{\partial F}(q)$. Taking the duality product both sides in \eqref{duality product1} we obtain
\begin{equation}\label{duality product2}
    \langle {}^C D_t^\alpha q(t), v^*\rangle=\langle Aq(t)+Bu(t)+Hf(t),v^*\rangle~~ ~\text{a.a.}~t\in I,~\text{for all}~v^*\in X^*.
\end{equation}
As $f(t)\in \partial F(t, q(t))$, from the definition of the Clarke subdifferential we obtain
\begin{equation}
	\langle f(t), v\rangle_{X^*,X}\le F^0(t, q(t);v), \forall v\in X.
\end{equation}
In particular, for $v=H^*v^*$ where $v^*\in X^*$ we obtain
\begin{equation}
	\langle f(t), H^*v^*\rangle_{X^*,X}\le F^0(t, q(t);H^*v^*),
\end{equation}
which further implies
\begin{equation}\label{HV4}
	\langle Hf(t), v^* \rangle_{X,X^*}\le F^0(t, q(t); H^*v^*).
\end{equation}

With the help of \eqref{HV4} we obtain from \eqref{duality product2} that $q\in C(I,X)$ satisfies
\begin{equation}
	\langle -^CD_t^{\alpha} q(t)+Aq(t)+Bu(t), v^*\rangle+F^0(t,q(t);H^*v^*)\ge 0, ~\text{for all}~v^*\in X^*,~\text{for}~t\in I.
\end{equation} 
In conclusion, according to the definition of the Clarke subdifferential, the solutions of the fractional differential inclusion presented in \eqref{PC1} is also a solution to the fractional evolution hemivariational inequality outlined in \eqref{HV1}. Consequently, our focus primarily is on addressing the inclusion presented in \eqref{PC1} to establish the approximate controllability of the fractional evolution hemivariational inequality \eqref{HV1}.

Following the paper \cite{MR1903295} we now recall the definition of mild solution of the fractional problem \eqref{NHGS}.
\begin{definition}
Fix $u(\cdot)\in L^2(I,U)$. A function $q\in C(I,X)$ is said to be a mild solution of \eqref{NHGS} if there exists $f\in S_F(q)$ such that 
    \begin{equation}
    q(t)=S_\alpha(t)x_0 + \int^{t}_{0}(t-s)^{\alpha-1}T_\alpha(t-s)f(s)ds+\int^{t}_{0}(t-s)^{\alpha-1}T_\alpha(t-s)Bu(s)ds,~t\in I,
    \end{equation}
    where
    \begin{equation}\label{Salpha}
        S_{\alpha}(t)=\int^{\infty}_{0}\xi_{\alpha}(\tau)T(t^{\alpha}\tau)d\tau; 
    \end{equation}
    \begin{equation}\label{Talpha}
        T_{\alpha}(t)=\alpha\int^{\infty}_{0}\tau \xi_{\alpha}(\tau)T(t^{\alpha}\tau)d\tau;
    \end{equation}
    \begin{equation}
         \xi_{\alpha}(\tau)=\frac{1}{\alpha}\tau^{-1-\frac{1}{\alpha}}\overline{w_{\alpha}}(\tau^{-\frac{1}{\alpha}})\geq 0;
    \end{equation}
    where
    \begin{equation}
        \overline{w_{\alpha}}(\tau)=\frac{1}{\pi}\sum^{\infty}_{n=1}(-1)^{n-1}\tau^{-\alpha-1}\frac{\Gamma(n\alpha+1)}{n!}sin(n\pi\alpha),~ \tau\in(0,\infty),
    \end{equation}
       and  $\xi_{\alpha}$ is a probability density function defined on $(0,\infty)$ , that is \begin{equation}
           \xi_{\alpha}(\tau)\geq0,~ \tau \in (0,\infty)~\text{and}~ \int^{\infty}_{0}\xi_{\alpha}(\tau)d\tau=1.
       \end{equation}
\end{definition}
The following Lemma \cite[Lemma 3.2-3.4]{MR2579471} characterizes some properties of the operators $S_{\alpha}(t)$ and $T_{\alpha}(t)$.
\begin{lemma}\label{Bound of Salpha and Talpha}
The operators $S_{\alpha}(t)$ and $T_{\alpha}(t)$ have the following properties:
\begin{itemize}
    \item[(i)] For any fixed $t\geq 0$ , the operators $S_{\alpha}(t)$ and $T_{\alpha}(t)$ are linear and bounded. Moreover 
    \begin{equation}
        \norm{S_\alpha(t)}_{\mathcal{L}(X)}\leq M 
    \end{equation}
    and
    \begin{equation}
        \norm{T_\alpha(t)}_{\mathcal{L}(X)}\leq \frac{M}{\Gamma(\alpha)}. 
    \end{equation}
    \item[(ii)] The operators $S_{\alpha}(t)$ and $T_{\alpha}(t)$ are strongly continuous for $t\geq 0$.
    \item[(iii)] If $T(t)$ is compact for $t>0$ , then the operators $S_{\alpha}(t)$ and $T_{\alpha}(t)$ are also compact for $t>0$.
\end{itemize}
\end{lemma}
\subsection{Novelty and Idea of the Proof}
We provide detailed proof of the approximate controllability result in Theorem \ref{MR} given in Subsection 3.3. Here, we sketch the idea and novelty of the proof. We prove that the problem \eqref{PC1} is approximately controllable, that means for any $\epsilon>0$, and for any desired final state $z\in X$, we can find a mild solution $q\in C(I, X)$ of \eqref{PC1} corresponding to a suitable control $u\in L^2(I,U)$ such that 
\begin{equation}
	\norm{q(a)-z}<\epsilon.
\end{equation}
 Denote $G(t,x)=\partial F(t,x), t\in I, x\in X$. 
The usual approach to prove the approximate controllability result is the following: let $z\in X$ be the final state we want to achieve in time $a$. For each $\epsilon>0$,  define a multivalued map $\Lambda_{\epsilon}:C(I,X)\multimap C(I,X)$ as follows: if $y\in \Lambda_{\epsilon}(q), q\in C(I,X)$, then 
\begin{equation}
	y(t)=S_{\alpha}(t)x_0+\int_{0}^{t}(t-s)^{\alpha-1}T_{\alpha}(t-s)Hg(s)ds+\int_{0}^{t}(t-s)^{\alpha-1}T_{\alpha}(t-s)Bu(s)ds, ~t\in I,
\end{equation}   
where $g\in S_G(q)$, and $u\in L^2(I,U)$ is given by
\begin{equation}\label{co}
	u(t)=B^*T_{\alpha}^*(a-t)J\left((\epsilon I+R_{\alpha}(a)J)^{-1}\left(z-S_{\alpha}(a)x_0-\int_{0}^{a}(a-s)^{\alpha-1}T_{\alpha}(a-s)Hg(s)ds\right)\right).
\end{equation}
The mapping $J: X\multimap X^*$ denotes the duality mapping, which we can define by
\begin{equation}\label{4.5}
	J(y)=\{y^*\in X^*: \langle y, y^*\rangle=\norm{y}_X^2=\norm{y^*}^2_{X^*}\}, ~\forall y\in X.
\end{equation}
If $X$ is a super-reflexive Banach space, by Proposition \ref{Propos} (see Appendix), $X$ can be renormed in such a way that $X$ becomes uniformly smooth. Consequently, the mapping $J$ becomes single-valued and uniformly continuous on bounded subsets of $X$ (see \cite[Proposition 32.22]{MR1033498}). If the mapping $J$ is single-valued, we say $J$ is normalized duality mapping.

Observe that the fixed points of this multimap $\Lambda_{\epsilon}$ are the mild solutions of the problem \eqref{PC1}. It is worth noting that the set $\Lambda_{\epsilon}(q)$ fails to be convex due to the nonlinear nature of the duality mapping in reflexive or super-reflexive Banach spaces (see \cite[Proposition 3.14]{MR2504478} for duality mappings in $L^p([0,1])$ spaces). However, as far as we know, most of the multivalued fixed point theorems required to provide existence results, the fixed point map must have the condition that its values are convex. Hence, these fixed point theorems are inapplicable to deduce the existence of fixed points of the map $\Lambda_{\epsilon}$. To overcome this issue we introduce a new map $\Upsilon_{\epsilon}: L^{1}(I,X^*)\to C(I,X)$ as follows: for $g\in L^{1}(I,X^*)$ we define
\begin{equation}
	(\Upsilon_{\epsilon} g)(t)=S_{\alpha}(t)x_0+\int_{0}^{t}(t-s)^{\alpha-1}T_{\alpha}(t-s)Hg(s)ds+\int_{0}^{t}(t-s)^{\alpha-1}T_{\alpha}(t-s)Bu(s)ds, ~t\in I,
\end{equation}
where $u\in L^2(I,U)$ is given by \eqref{co}. We prove that if the space $X$ is super-reflexive, then $\Upsilon_{\epsilon}$ is continuous. We next define $S_G:C(I,X)\multimap L^{1}(I,X^*)$ as in \eqref{Selection} replacing $\partial F$ with $G$. Then define $\Gamma_{\epsilon}: L^{1}(I,X^*)\multimap L^1(I,X^*)$ as $\Gamma_{\epsilon}(g)=S_G(\Upsilon_{\epsilon}(g))$. From Lemma \ref{Lemma 2.14} it is clear that the multimap $\Gamma_{\epsilon}$ has convex values. Therefore, defining such a map we can tackle the convexity issue and establish the controllability result. This is the main novelty of this manuscript.
\subsection{Some basic results about Controllability}
We now turn into the topic of approximate controllability of the problem \eqref{PC1}. To analyze the approximate controllability of Problem \eqref{PC1}, we consider the fractional linear system \eqref{Linear} associated with Problem \eqref{PC1}.
With no claim of originality, we briefly introduce some well-known results about the approximate controllability of the fractional linear control system  \eqref{Linear}.
Considering the initial condition $x_0=0$, we define the bounded linear map
$W:L^2(I,U)\to X$ by
\begin{equation*}
	W(u)=\int_{0}^{a} (a-s)^{\alpha-1}T_{\alpha}(a-s)Bu(s)ds, ~u\in L^2(I,U).
\end{equation*}
In the above, the operator $T_{\alpha}(t), ~t\in I$ is given in \eqref{Talpha}. 
\begin{definition}\label{Def4.1}
	The system \eqref{Linear} is approximately controllable on $I$ if $\overline{\operatorname{Range}(W)}=X$, where overline denotes the closure in $X$.
\end{definition}
We now define the operator $R_{\alpha}(t), ~t\in I$ as follows:
\begin{equation}
    R_{\alpha}(t)x^*=\int_{0}^{t}(t-s)^{\alpha-1}T_{\alpha}(t-s)BB^*T_{\alpha}^*(a-s)x^*ds, ~x^*\in X^*.
\end{equation}
In particular, we define the controllability Gramian operator $R_{\alpha}(a):X^*\to X$ as follows:
\begin{equation}\label{Ralpha}
    R_{\alpha}(a)x^*=\int_{0}^{a}(a-s)^{\alpha-1}T_{\alpha}(a-s)BB^*T_{\alpha}^*(a-s)x^*ds, ~x^*\in X^*.
\end{equation}

As we assume $\frac{1}{2}<\alpha<1$, the operator $R_{\alpha}(t)$ is well defined. In fact, for $x^*\in X^*$ we estimate
\begin{align*}
    \norm{R_{\alpha}(t)x^*}_X\le &\int_{0}^{t}(t-s)^{\alpha-1}\norm{T_{\alpha}(t-s)BB^*T_{\alpha}^*(a-s)x^*}ds\\
    \le & \frac{M^2}{[\Gamma(\alpha)]^2}\norm{B}^2\norm{x^*}_{X^*}\int_{0}^{t}(t-s)^{\alpha-1}ds\\
    \le & \frac{M^2}{[\Gamma(\alpha)]^2}\norm{B}^2\norm{x^*}_{X^*}\frac{t^{\alpha}}{\alpha}.
\end{align*}
Hence we obtain
\begin{align}\label{R(a) bounded}
    \norm{R_{\alpha}(t)x^*}_X
    \le & \frac{M^2}{[\Gamma(\alpha)]^2}\norm{B}^2\norm{x^*}_{X^*}\frac{t^{\alpha}}{\alpha}, ~t\in I.
\end{align}
Moreover, it is easy to see that the operator $R_{\alpha}(t):X^*\to X$ is linear and from the estimate \eqref{R(a) bounded} we confirm that the operator $R_{\alpha}(t)$ is also bounded.

In the following Lemmas we deduce some more properties of the operator $R_{\alpha}(a)$.
\begin{lemma}
$R_{\alpha}(a):X^*\rightarrow X$ is symmetric operator i.e., $\langle x^*_1,R_{\alpha}(a)x^*_2 \rangle = \langle x^*_2 , R_{\alpha}(a)x^*_1\rangle $  for all $x^*_1 , x^*_2 \in X^*$.
\end{lemma}
\begin{proof}
Let $x_1^*, x_2^*\in X^*$. We estimate
\begin{align*}
\langle x^*_1,R_{\alpha}(a)x^*_2 \rangle _{X^*,X}=& \left\langle x^*_1 , \int ^{a}_{0}(a-s)^{\alpha-1}T_{\alpha}(a-s)BB^*T^*_{\alpha}(a-s)x^*_2ds\right\rangle \\
 =&\int ^{a}_{0}(a-s)^{\alpha-1}\langle x^*_1,T_{\alpha}(a-s)BB^*T^*_{\alpha}(a-s)x^*_2
 \rangle ds \\    
 =&\int ^{a}_{0}(a-s)^{\alpha-1}\langle T_{\alpha}(a-s)BB^*T^*_{\alpha}(a-s)x^*_1,x^*_2
 \rangle ds\\
 =&\left\langle \int ^{a}_{0}(a-s)^{\alpha-1}T_{\alpha}(a-s)BB^*T^*_{\alpha}(a-s)x^*_1ds, x^*_2\right\rangle.
\end{align*}
Using the definition of the operator $R_{\alpha}(a)$ we obtain
\begin{align*}
\langle x^*_1,R_{\alpha}(a)x^*_2 \rangle _{X^*,X}
 = & \langle x_2^*, R_{\alpha}(a)x_1^*\rangle_{X^*,X}.
\end{align*}
Hence, we conclude that the linear operator $R_{\alpha}(a)$ is a symmetric operator.
\end{proof}

The following lemma is of crucial importance.
 \begin{lemma}\label{Lem3.1}\cite[Lemma 2.2]{MR2046377}
 	Suppose that $X$ is a separable reflexive Banach space. Then the map $\epsilon I+R_{\alpha}(a)J$ is invertible for every $\epsilon>0$, and satisfies the following estimates
 	\begin{equation*}
 		\norm{\epsilon(\epsilon I+R_{\alpha}(a)J)^{-1}y}\le\norm{y},
	\end{equation*}
	for all $y\in X$. In the above, $J: X\to X^*$ is the normalized duality mapping.
 \end{lemma}
 Note that if $X$ is a reflexive Banach space, then by \cite{MR222610}, $X$ can be renormed such that $X$ and $X^*$ become strictly convex. As a consequence, the duality mapping $J$  becomes single-valued.

We now give a criterion for the approximate controllability of the system \eqref{Linear} in terms of the adjoint of the control operator $B$.
\begin{theorem}\label{ACOF}
The following statements are equivalent :
\begin{itemize}
    \item[(i)] The fractional linear control system \eqref{Linear} is approximately controllable on $I=[0,a]$.
    \item[(ii)] For $x^* \in X^*$, we have \begin{equation*}
		B^*T_\alpha^*(a-t)x^*=0, ~\text{for all }~t\in I \Rightarrow x^*=0.
	\end{equation*}
 \item[(iii)] $\forall x \in X$ , $\epsilon(\epsilon I + R_{\alpha}(a)J)^{-1}(x)\rightarrow 0$ as $\epsilon \rightarrow 0^+$ in strong topology.
\end{itemize}
\end{theorem}
\begin{lemma}
Suppose the fractional linear control sysetm \eqref{Linear} is approximately controllable in $I$. Then the operator $R_{\alpha}(a):X^*\rightarrow X$ is a positive operator.
\end{lemma}
\begin{proof}
For $x^*\in X^*$ we estimate
\begin{align*}
        \langle x^* , R_{\alpha}(a)x^*\rangle = &\left\langle x^* , \int ^{a}_{0}(a-s)^{\alpha-1}T_{\alpha}(a-s)BB^*T^*_{\alpha}(a-s)x^*ds\right\rangle\\
        = & \int ^{a}_{0}(a-s)^{\alpha-1}\langle x^* , T_{\alpha}(a-s)BB^*T^*_{\alpha}(a-s)x^*\rangle ds\\
        = & \int ^{a}_{0}(a-s)^{\alpha-1}\langle B^*T^*_{\alpha}(a-s)x^* , B^*T^*_{\alpha}(a-s)x^*\rangle ds\\
        = & \int ^{a}_{0}(a-s)^{\alpha-1}\norm{B^*T^*_{\alpha}(a-s)x^*}^2ds\geq 0 
        \end{align*}
From this, we conclude that the operator $R_{\alpha}(a)$ is a non negative operator. 

Moreover, by means of Theorem \ref{ACOF}, we confirm that $\langle R_{\alpha}(a)x^*, x^*\rangle=0$ implies $x^*=0$. Hence, the operator $R_{\alpha}(a)$ is a positive operator.
\end{proof}

The following simple observation is crucial, and it is immediate from the positivity of the operator $R_{\alpha}(a)$.

\begin{lemma}\label{Lem4.5}
	Suppose the system \eqref{Linear} is approximately controllable on $I$; then $R_{\alpha}(a)$ is injective.
\end{lemma}

The following Theorem is useful.
\begin{theorem}\cite[Lemma 4.4]{MR4104454}\label{Thm6.0}
	Let $U$ be a separable Hilbert space and $X$ be a Banach space with dual $X^*$. Assume that the system 
	\begin{equation}
		\begin{cases}
			^CD^{\alpha}_{t}q(t)=Aq(t)+Bu(t), t\in I;\\
			q(0)=x_0
		\end{cases}
	\end{equation} 
	is approximately controllable on $I$. Additionally assume that there exists a relatively compact set $K\subset X$, and $x_{\epsilon} \in X$ such that
	\begin{equation}
		x_{\epsilon}\in \epsilon(\epsilon I+R_{\alpha}(a)J)^{-1}(K), \epsilon>0,
	\end{equation}
	where $R_{\alpha}(a)$ is defined in \eqref{Ralpha} and $J:X\to X^*$ is the normalized duality mapping.
	Then there is a sequence $\epsilon_n\to 0$ as $n\to \infty$ such that $x_{\epsilon_n}\to 0$ as $n\to \infty$.
\end{theorem}
In the following Lemma, we demonstrate the continuity of the operator $(\epsilon I+R_{\alpha}(a)J)^{-1}: X\to X$. The proof can be found in the paper \cite{MR4336468}.
\begin{lemma}\label{Con}
	Let $X$ be a super-reflexive Banach space. The operator $(\epsilon I+R_{\alpha}(a)J)^{-1}: X\to X$ is uniformly continuous in every bounded subset of $X$, where $R_{\alpha}(a):X^*\to X$ is as in \eqref{Ralpha} and $J:X\to X^*$ is the normalized duality mapping.  
\end{lemma}
We now define the measurable selection map $S_G: C(I,X)\multimap L^1(I,X^*)$ as
\begin{equation}
	S_G(q)=\{g\in L^1(I,X^*): g(t)\in G(t,q(t))~\text{a.a.}~t\in I\}.
\end{equation}
The following Theorem says that the set $S_G(q)$ is nonempty for every $q\in C(I, X)$ and delivers some property.
\begin{theorem}\cite[Lemma 5.3]{MR2976197}\label{Thm3.2}
	Suppose the multivalued map $G:I\times X \multimap X^*$ satisfies the hypotheses (G1)-(G3).
	Then, the following is true.
	\begin{itemize}
		\item[(i)] the multifunction $S_G$ has nonempty and weakly compact convex values.
		\item[(ii)] the multimap $S_G$ has strongly weakly closed graph in the following sense: suppose $g_n\in S_G(q_n)$ with $g_n\rightharpoonup g$ in $L^1(I, X^*)$, $q_n\to q$ in $C(I, X)$, then $g\in S_G(q)$. 
	\end{itemize} 
\end{theorem}
\subsection{Main Result- Approximate Controllability}
With the results mentioned above and hypotheses in mind, we deliver the main result of this paper, the approximate controllability result for the Problem \eqref{PC1}.
\begin{theorem}\label{MR}
	Assume the Hypotheses (T) and (U) hold. Assume that the multivalued nonlinearity $G$ satisfies Hypotheses (G1)-(G3). Then, the problem \eqref{PC1} is approximately controllable in $I$.
\end{theorem}
\begin{proof}
Define the sets 
	\begin{equation}
		B_C^N=\{q\in C(I,X): \norm{q}_{C(I,X)}\le N\}, N>0,
	\end{equation}
and
	\begin{equation}
		Q_{G}=\{g\in  L^1(I, X^*): \norm{g(t)}\le \eta(t)~\text{a.a.}~t\in I\},
	\end{equation}
where $\eta$ is as in Hypothesis (G3).
	
Fix $\epsilon>0$ and let $z\in X$ be the final state we want to achieve in time $a$. We define the  map $\Upsilon_{\epsilon}: Q_G\to C(I,X)$ as follows: for $g\in Q_G$, $\Upsilon_{\epsilon}(g)$ satisfies
	\begin{equation}\label{Y1}
		\Upsilon_{\epsilon}(g)(t)=S_\alpha(t)x_0+\int_{0}^{t}(t - s)^{\alpha - 1} T_\alpha(t-s)Hg(s)ds+\int_{0}^{t}(t - s)^{\alpha - 1}T_\alpha(t-s)Bu(s)ds, ~t\in I,
	\end{equation}   
	where $u\in L^2(I,U)$ is given by
	\begin{equation}\label{Y2}
		u(t)=B^*T_{\alpha}^*(a-t)J\left((\epsilon I+R_{\alpha}(a)J)^{-1}\left(z-S_\alpha(a)x_0-\int_{0}^{a}(a - s)^{\alpha - 1}T_{\alpha}(a-s)Hg(s)ds\right)\right).
	\end{equation}

	We prove this Theorem in several steps.\\
	\textbf{(STEP-I)} 	We now prove the map $\Upsilon_{\epsilon}$ maps the set $Q_G$ into the set $B_C^{N_0}$ for some $N_0$. Let $y=\Upsilon_{\epsilon}(g), g\in Q_G$. Then $y$ satisfies \eqref{Y1}, \eqref{Y2}. With the help of expression \eqref{Y1} together with Lemma \ref{Bound of Salpha and Talpha} we estimate 
	\begin{align}\label{Yec}
		\norm{y(t)}\le & M\norm{x_0}+\frac{M}{\Gamma(\alpha)}\int_{0}^{t}(t - s)^{\alpha - 1}\norm{H}_{\mathcal{L}(X^*,X)}\norm{g(s)}ds+\norm{\int_{0}^{t}(t - s)^{\alpha - 1}T_{\alpha}(t-s)Bu(s)ds}\\
  \le &M\norm{x_0}+\frac{M}{\Gamma(\alpha)}\norm{H}_{\mathcal{L}(X^*,X)}\int_{0}^{t}(t - s)^{\alpha - 1}\eta(s)ds+\norm{\int_{0}^{t}(t - s)^{\alpha - 1}T_{\alpha}(t-s)Bu(s)ds}\
	\end{align}
 By employing H$\ddot{\text{o}}$lder's inequality, we obtain from above
 \begin{align}
 \norm{y(t)}\le &M\norm{x_0}+\frac{M}{\Gamma(\alpha)}\norm{H}_{\mathcal{L}(X^*,X)}\norm{\eta}_{L^{\frac{1}{\alpha_1}}}\left(\frac{t^{\beta(\alpha-1)+1}}{\beta(\alpha-1)+1}\right)^{1-\alpha_1}+\norm{\int_{0}^{t}(t - s)^{\alpha - 1}T_{\alpha}(t-s)Bu(s)ds}\\
 \le &M\norm{x_0}+\frac{M}{\Gamma(\alpha)}\norm{H}_{\mathcal{L}(X^*,X)}\norm{\eta}_{L^{\frac{1}{\alpha_1}}}\left(\frac{a^{\beta(\alpha-1)+1}}{\beta(\alpha-1)+1}\right)^{1-\alpha_1}+\norm{\int_{0}^{t}(t - s)^{\alpha - 1}T_{\alpha}(t-s)Bu(s)ds},
	\end{align}
 where $\beta=\frac{1}{1-\alpha_1}$. Denote by $\Theta=\left(\frac{a^{\beta(\alpha-1)+1}}{\beta(\alpha-1)+1}\right)^{1-\alpha_1}\norm{\eta}_{L^{\frac{1}{\alpha_1}}}$. With this notation, we can write the above expression as
  \begin{align}\label{YY}
 \norm{y(t)}
 \le &M\norm{x_0}+\frac{M}{\Gamma(\alpha)}\norm{H}_{\mathcal{L}(X^*,X)}\Theta+\norm{\int_{0}^{t}(t - s)^{\alpha - 1}T_{\alpha}(t-s)Bu(s)ds}.
	\end{align}
	Using the expression \eqref{Y2} we estimate
 \begin{align*}
     &	\norm{\int_{0}^{t}(t - s)^{\alpha - 1}T_\alpha(t-s)Bu(s)ds}\\
		=&\norm{\int_{0}^{t}(t - s)^{\alpha - 1}T_\alpha(t-s)BB^*T_\alpha^*(a-s)J\right.\\
  &\hspace{3cm}\left.\times\left((\epsilon I+R_{\alpha}(a)J)^{-1}\left(z-S_\alpha(a)x_0-\int_{0}^{a}(a - r)^{\alpha - 1}T_\alpha(a-r)Hg(r)dr\right)\right)ds}\\
 \end{align*}
 Invoking Lemma \ref{Bound of Salpha and Talpha} and using the definition of duality mapping, we obtain
 \begin{align*}
     &	\norm{\int_{0}^{t}(t - s)^{\alpha - 1}T_\alpha(t-s)Bu(s)ds}\\
		\le & \frac{1}{\epsilon}\frac{M^2\norm{B}^2}{(\Gamma(\alpha))^2}\int_{0}^{t}(t - s)^{\alpha - 1}\norm{J\left(\epsilon(\epsilon I+R_{\alpha}(a)J)^{-1}\left(z-S_\alpha(a)x_0-\int_{0}^{a}(a - r)^{\alpha - 1}T_\alpha(a-r)Hg(r)dr\right)\right)}ds\\
		= & \frac{1}{\epsilon}\frac{M^2\norm{B}^2}{(\Gamma(\alpha))^2}\int_{0}^{t}(t-s)^{\alpha-1}\norm{\epsilon(\epsilon I+R_{\alpha}(a)J)^{-1}\left(z-S_\alpha(a)x_0-\int_{0}^{a}(a - r)^{\alpha - 1}T_\alpha(a-r)Hg(r)dr\right)}ds\\
 \end{align*}
 By virtue of Lemma \ref{Lem3.1}, we obtain
  \begin{align}
     &	\norm{\int_{0}^{t}(t - s)^{\alpha - 1}T_\alpha(t-s)Bu(s)ds}\\
  \le & \frac{1}{\epsilon}\frac{M^2\norm{B}^2}{(\Gamma(\alpha))^2}\int_{0}^{t}(t-s)^{\alpha-1}\norm{z-S_\alpha(a)x_0-\int_{0}^{a}(a - r)^{\alpha - 1}T_\alpha(a-r)Hg(r)dr}ds\\
  \le & \frac{1}{\epsilon}\frac{M^2\norm{B}^2}{(\Gamma(\alpha))^2} \int_{0}^{t}(t-s)^{\alpha-1}[\norm{z}+ M\norm{x_0}+\frac{M}{\Gamma(\alpha)}\norm{H}_{\mathcal{L}(X^*,X)}\int_{0}^{a}(a - r)^{\alpha - 1}\norm{g(r)}dr]ds \\
   \le & \frac{1}{\epsilon}\frac{M^2\norm{B}^2}{(\Gamma(\alpha))^2} \frac{a^{\alpha}}{\alpha}\left[\norm{z}+M\norm{x_0}+\frac{M}{\Gamma(\alpha)}\norm{H}_{\mathcal{L}(X^*,X)}\Theta\right].
 \end{align}
In the last inequality we use Hypothesis (G3).
 Therefore, inserting the value of $\displaystyle \norm{\int_{0}^{t}(t-s)^{\alpha-1}T_\alpha(t-s)Bu(s)ds}$ in \eqref{YY} we obtain
	\begin{align*}
		\norm{y(t)}\le & M\norm{x_0}+\frac{M}{\Gamma(\alpha)}\norm{H}_{\mathcal{L}(X^*,X)}\Theta+\frac{1}{\epsilon}\frac{M^2\norm{B}^2}{(\Gamma(\alpha))^2} \frac{a^{\alpha}}{\alpha}\left[\norm{z}+M\norm{x_0}+\frac{M}{\Gamma(\alpha)}\norm{H}_{\mathcal{L}(X^*,X)}\Theta\right]=N_0~\text{(say)}.
	\end{align*}
	From this we conclude that there exists $N_0$ (dependent on $\epsilon$) such that $\Upsilon_{\epsilon}$ maps the set $Q_G$ into $B_C^{N_0}$.\\
	
	\textbf{(STEP-II)} In this step we show that the multimap $\Upsilon_{\epsilon}$ is continuous from the set $Q_G$ endowed with the weak topology into $C(I,X)$. Observe that the set $Q_G$ is a compact convex metrizable set in $w-L^1(I, X^*)$. Therefore, it suffices to prove the sequential continuity of the operator $\Upsilon_{\epsilon}$. Assume that a sequence $\{g_n\}_{n\in \mathbb{N}}\subset Q_G$, converges to $g$ in the space $w-L^1(I, X^*), q_n=\Upsilon_{\epsilon}(g_n), n\in \mathbb{N}$ and $q=\Upsilon_{\epsilon}(g)$. We show $q_n\to q$ in $C(I,X)$.
	By the definition of the operator $\Upsilon_{\epsilon}$, we obtain
	\begin{equation}\label{4.7}
		q_n(t)=S_\alpha(t)x_0+\int_{0}^{t}(t-s)^{\alpha -1}T_\alpha(t-s)Hg_n(s)ds+\int_{0}^{t}(t-s)^{\alpha - 1} T_\alpha(t-s)Bu_n(s)ds, ~t\in I,~ n\in \mathbb{N},
	\end{equation}
	and
	\begin{equation}
		q(t)=S_\alpha(t)x_0+\int_{0}^{t}(t-s)^{\alpha-1}T_\alpha(t-s)Hg(s)ds+\int_{0}^{t}(t-s)^{\alpha-1}T_\alpha(t-s)Bu(s)ds, ~t\in I.
	\end{equation}
	In the above,
	\begin{equation}
		u_n(t)=B^*T_\alpha^*(a-t)J\left((\epsilon I+R_{\alpha}(a)J)^{-1}\left(z-S_\alpha(a)x_0-\int_{0}^{a}(a-s)^{\alpha -1}T_\alpha(a-s)Hg_n(s)ds\right)\right),~ t\in I,
	\end{equation}
	and
	\begin{equation}
		u(t)=B^*T_\alpha^*(a-t)J\left((\epsilon I+R_{\alpha}(a)J)^{-1}\left(z-S_\alpha(a)x_0-\int_{0}^{a}(a-s)^{\alpha-1}T_{\alpha}(a-s)Hg(s)ds\right)\right),~ t\in I.
	\end{equation}
	As the map $H:X^*\to X$ is bounded linear, we can define the map $\Phi:L^1(I,X^*)\to L^1(I,X)$ as $\Phi(g)(s)=Hg(s),~ s\in I,~ g\in L^1(I,X^*)$. For $\phi$ belongs to the dual of $L^1(I,X)$, we define the map $\Psi: L^1(I,X^*)\to \mathbb{R}$ as $\Psi=\phi \circ \Phi$. It is obvious that the map $\Psi$ belongs to the dual of $L^1(I, X^*)$. As $g_n\rightharpoonup g$ in $L^1(I,X^*)$, by the definition of the weak convergence we conclude that
	\begin{equation}
		\phi(\Phi(g_n))=\Psi(g_n)\to \Psi(g)=\phi(\Phi(g)),~ \text{as}~n\to\infty,
	\end{equation}
	from which we conclude that $\Phi(g_n)\rightharpoonup \Phi(g)$ in $L^1(I,X)$. The compactness of the semigroup $\{T_\alpha(t)\}_{t\ge 0}$ together with Lemma \ref{Lem3.6} imply that 
	\begin{equation}
		\int_{0}^{t}(t-s)^{\alpha-1}T_\alpha(t-s)\Phi(g_n)(s)ds\to \int_{0}^{t}(t-s)^{\alpha-1}T_\alpha(t-s)\Phi(g)(s)ds,~\text{uniformly in}~C(I,X),
	\end{equation}
	and consequently, we obtain
	\begin{equation}\label{4.}
		\int_{0}^{t}(t-s)^{\alpha-1}T_\alpha(t-s)Hg_n(s)ds\to \int_{0}^{t}(t-s)^{\alpha-1}T_\alpha(t-s)Hg(s)ds,~\text{uniformly in}~C(I,X).
	\end{equation}
	In particular,
	\begin{equation}
		\int_{0}^{a}(a-s)^{\alpha-1}T_\alpha(a-s)Hg_n(s)ds\to \int_{0}^{a}(a-s)^{\alpha-1}T_\alpha(a-s)Hg(s)ds~\text{as}~n\to \infty~\text{in}~X.
	\end{equation}
	By the continuity of the operator $(\epsilon I+R_{\alpha}(a)J)^{-1}$ [ See Lemma \ref{Con}] we obtain
	\begin{align*}
		&u_n(t)=B^*T_\alpha^*(a-t)J\left((\epsilon I+R_{\alpha}(a)J)^{-1}\left(z-S_\alpha(a)x_0-\int_{0}^{a}(a-s)^{\alpha-1}T_\alpha(a-s)Hg_n(s)ds\right)\right)\\
		&\to B^*T_\alpha^*(a-t)J\left((\epsilon I+R_{\alpha}(a)J)^{-1}\left(z-S_\alpha(a)x_0-\int_{0}^{a}(a-s)^{\alpha-1}T_\alpha(a-s)Hg(s)ds\right)\right)~\text{as}~n\to \infty.
	\end{align*}
	Further, using the definition of the duality map together with Lemma \ref{Lem3.1}, we compute
	\begin{align*}
		\norm{u_n(t)}\le& \frac{M}{\Gamma(\alpha)}\norm{B}\norm{J\left((\epsilon I+R_{\alpha}(a)J)^{-1}\left(z-S_\alpha(a)x_0-\int_{0}^{a}(a-s)^{\alpha-1}T_\alpha(a-s)Hg_n(s)ds\right)\right)}\\
		\le& \frac{M}{\Gamma(\alpha)}\norm{B}\norm{(\epsilon I+R_{\alpha}(a)J)^{-1}\left(z-S_\alpha(a)x_0-\int_{0}^{a}(a-s)^{\alpha-1}T_\alpha(a-s)Hg_n(s)ds\right)}\\
		\le&\frac{1}{\epsilon} \frac{M}{\Gamma(\alpha)}\norm{B}\norm{z-S_\alpha(a)x_0-\int_{0}^{a}(a-s)^{\alpha-1}T_\alpha(a-s)Hg_n(s)ds}\\
		\le&\frac{1}{\epsilon} \frac{M}{\Gamma(\alpha)}\norm{B}\left[\norm{z}+M\norm{x_0}+\frac{M}{\Gamma(\alpha)}\norm{H}_{\mathcal{L}(X^*,X)}\Theta\right], t\in I.
	\end{align*}
	Employing \cite[Theorem 7.2]{bartle2014elements} we conclude that $u_n\to u$ in $L^2(I,U)$.  We now estimate
	\begin{align*}
		&\norm{\int_{0}^{t}(t-s)^{\alpha-1}T_\alpha(t-s)Bu_n(s)ds-\int_{0}^{t}(t-s)^{\alpha-1}T_\alpha(t-s)Bu(s)ds}\\
		=&\norm{\int_{0}^{t}(t-s)^{\alpha-1}T_\alpha(t-s)[Bu_n(s)-Bu(s)]ds}\\
		\le& \frac{M}{\Gamma(\alpha)}\norm{B}\int_{0}^{t}(t-s)^{\alpha-1}\norm{u_n(s)-u(s)}ds\\
		\le &\frac{M}{\Gamma(\alpha)}\norm{B}\left(\frac{a^{2\alpha-1}}{2\alpha-1}\right)^{\frac{1}{2}}\norm{u_n-u}_{L^2(I,U)},
	\end{align*}
	where in the last inequality we use H$\ddot{\text{o}}$lder's inequality and recall $\frac{1}{2}<\alpha<1$.
	From this, we conclude that 
	\begin{equation}\label{4.13}
		\int_{0}^{t}(t-s)^{\alpha-1}T_\alpha(t-s)Bu_n(s)ds\to \int_{0}^{t}(t-s)^{\alpha-1}T_\alpha(t-s)Bu(s)ds~\text{as}~n\to \infty~\text{uniformly in}~C(I,X).
	\end{equation}
	Therefore, passing limit in \eqref{4.7} as $n\to \infty$ and using \eqref{4.} and \eqref{4.13} we obtain
	\begin{equation}
		q_n(t)\to q(t)~\text{uniformly in}~C(I,X)~\text{as}~n\to \infty.	
	\end{equation}
	This shows that the map $\Upsilon_{\epsilon}$ is continuous from $w-Q_{G}$ into $C(I, X)$.\\
	
	\textbf{(STEP-III)} We now define the multimap $\Gamma_{\epsilon}: Q_G\multimap L^{1}(I,X^*)$ as
	\begin{equation}
		\Gamma_{\epsilon}(g)=S_G(\Upsilon_{\epsilon}(g)), g\in Q_G.
	\end{equation}
	In the above the multimap $S_G: \Upsilon_{\epsilon}(Q_G)\multimap L^1(I,X^*)$ is defined by
	\begin{equation}
		S_G(q)=\{g\in L^1(I,X^*): g(t)\in G(t, q(t))~\text{a.a.}~t\in I\}, q\in \Upsilon_{\epsilon}(Q_G).
	\end{equation}
	Clearly the map $\Gamma_{\epsilon}$ is well defined and has weakly compact convex values. Also by Hypothesis (G3) it follows that $\Gamma_{\epsilon}(Q_G)\subset Q_G$.\\
	
	\textbf{(STEP-IV)} In this step we prove the multimap $\Gamma_{\epsilon}$ is upper semicontinuous in $w-L^{1}(I,X^*)$.  As the multimap $\Gamma_{\epsilon}$ is weakly compact, it is sufficient to prove that the multimap $\Gamma_{\epsilon}$ has a weakly closed graph.
	To prove this consider $\{g_n\}_{n\in \mathbb{N}}\subset Q_G, \{f_n\}_{n\in \mathbb{N}}\subset Q_G$ be such that $f_n\rightharpoonup f$, $g_n\rightharpoonup g$ in $L^1(I,X^*)$ and $f_n\in \Gamma_{\epsilon}(g_n)$ for all $n\in \mathbb{N}$. We show $f\in \Gamma_{\epsilon}(g)$. The fact that $f_n\in \Gamma_{\epsilon}(g_n)$ we obtain
	\begin{equation}
		f_n\in \Gamma_{\epsilon}(g_n)=S_G(\Upsilon_{\epsilon}(g_n)), g_n\in Q_G.
	\end{equation}
	As $g_n\rightharpoonup g$ in $L^1(I,X^*)$, and the map $\Upsilon_{\epsilon}$ is continuous from $w-Q_G$ into $C(I,X)$, we obtain
	\begin{equation}
		\Upsilon_{\epsilon}(g_n)\to \Upsilon_{\epsilon}(g) ~\text{in}~C(I,X).
	\end{equation}
	Also, by Theorem \ref{Thm3.2} the map $S_G$ has strongly weakly closed graph, hence $f\in S_G(\Upsilon_{\epsilon}(g))=\Gamma_{\epsilon}(g)$. Thus, $\Gamma_{\epsilon}$ has a weakly closed graph.\\

	\textbf{(STEP-V)} Therefore by Theorem \ref{fixed}, the map $\Gamma_{\epsilon}$ has a fixed point for each $\epsilon>0$. Therefore, we have 
	\begin{equation}
		g_{\epsilon}\in S_G(\Upsilon_{\epsilon}(g_{\epsilon})), g_{\epsilon}\in Q_G.
	\end{equation}
	Let $q_{\epsilon}=\Upsilon_{\epsilon}(g_{\epsilon})$, then $g_{\epsilon}\in S_G(q_{\epsilon})$. The definition of the map $\Upsilon_{\epsilon}$ gives 
	\begin{align}
		q_{\epsilon}(t)=S_\alpha(t)x_0+\int_{0}^{t}(t-s)^{\alpha-1}T_\alpha(t-s)Hg_{\epsilon}(s)ds+\int_{0}^{t}(t-s)^{\alpha-1}T_\alpha(t-s)Bu_{\epsilon}(s)ds, t\in I.
	\end{align}
	In the above $u_{\epsilon}\in L^2(I,U)$ is given by
	\begin{equation}
		u_{\epsilon}(t)=B^*T_\alpha^*(a-t)J\left((\epsilon I+R_{\alpha}(a)J)^{-1}\left(z-S_{\alpha}(a)x_0-\int_{0}^{a}(a-s)^{\alpha-1}T_\alpha(a-s)Hg_{\epsilon}(s)ds\right)\right).
	\end{equation}
	This concludes that $q_{\epsilon}$ solves problem \eqref{PC1} for each $\epsilon>0$.
	
	It remains to show that $\norm{q_{\epsilon}(a)-z}\to 0$ as $\epsilon\to 0$.
	
	As the set $Q_G$ is weakly compact in $L^1(I,X^*)$, we obtain $g_{\epsilon}\rightharpoonup g$ in $L^1(I,X^*)$ upto a subsequence. Therefore, using Lemma \ref{Lem3.6} and using the compactness of the semigroup, we obtain
	\begin{align}\label{conv}
		z-S_\alpha(a)x_0-\int_{0}^{a}(a-s)^{\alpha-1}T_\alpha(a-s)Hg_{\epsilon}(s)ds\to z-S_\alpha(a)x_0-\int_{0}^{a}(a-s)^{\alpha-1}T_\alpha(a-s)Hg(s)ds~\text{in}~X.
	\end{align}
	Recalling the definition of the map $R_{\alpha}(a)$, we now estimate
	\begin{align*}
		q_{\epsilon}(a)
		=&S_\alpha(a)x_0+\int_{0}^{a}(a-s)^{\alpha-1}T_\alpha(a-s)Hg_{\epsilon}(s)ds+\int_{0}^{a}(a-s)^{\alpha-1}T_\alpha(a-s)Bu_{\epsilon}(s)ds\\
		=&S_\alpha(a)x_0+\int_{0}^{a}(a-s)^{\alpha-1}T_\alpha(a-s)Hg_{\epsilon}(s)ds\\
		+&\int_{0}^{a}(a-s)^{\alpha-1}T_\alpha(a-s)BB^*T_{\alpha}^*(a-s)J\\
  &\hspace{2cm}\times\left((\epsilon I+R_{\alpha}(a)J)^{-1}\left(z-S_\alpha(a)x_0-\int_{0}^{a}(a-s)^{\alpha-1}T_\alpha(a-s)Hg_{\epsilon}(s)ds\right)\right)ds\\
		=&S_\alpha(a)x_0+\int_{0}^{a}(a-s)^{\alpha-1}T_\alpha(a-s)Hg_{\epsilon}(s)ds\\
		+&R_{\alpha}(a)J\left((\epsilon I+R_{\alpha}(a)J)^{-1}\left(z-S_{\alpha}(a)x_0-\int_{0}^{a}(a-s)^{\alpha-1}T_{\alpha}(a-s)Hg_{\epsilon}(s)ds\right)\right)\\
		=&S_\alpha(a)x_0+\int_{0}^{a}(a-s)^{\alpha-1}T_{\alpha}(a-s)Hg_{\epsilon}(s)ds\\
		+&(\epsilon I+R_{\alpha}(a)J-\epsilon I)(\epsilon I+R_{\alpha}(a)J)^{-1}\left(z-S_\alpha(a)x_0-\int_{0}^{a}(a-s)^{\alpha-1}T_\alpha(a-s)Hg_{\epsilon}(s)ds\right)\\
		=&S_\alpha(a)x_0+\int_{0}^{a}(a-s)^{\alpha-1}T_\alpha(a-s)Hg_{\epsilon}(s)ds+z-S_\alpha(a)x_0-\int_{0}^{a}(a-s)^{\alpha-1}T_\alpha(a-s)Hg_{\epsilon}(s)ds\\
		-&\epsilon(\epsilon I+R_{\alpha}(a)J)^{-1}\left(z-S_\alpha(a)x_0-\int_{0}^{a}(a-s)^{\alpha-1}T_\alpha(a-s)Hg_{\epsilon}(s)ds\right)\\
		=&z-\epsilon(\epsilon I+R_{\alpha}(a)J)^{-1}\left(z-S_\alpha(a)x_0-\int_{0}^{a}(a-s)^{\alpha-1}T_\alpha(a-s)Hg_{\epsilon}(s)ds\right)
	\end{align*}
	
	Therefore, by the convergence shown in \eqref{conv} together with Theorem \ref{Thm6.0} we obtain $q_{\epsilon}(a)\to z$ as desired. The proof is completed.
\end{proof}
\begin{remark}\label{RR}
	In proving the controllability result, it is required to assume that the space $X$ is super-reflexive. This requirement stems from the need for continuity of the map $\Upsilon_{\epsilon}$ within the proof. To establish the continuity of $\Upsilon_{\epsilon}$, we must demonstrate the continuity of the operator $(\epsilon I+R_{\alpha}(a)J)^{-1}: X\to X$. This, in turn, relies on $X$ possessing $p$-smoothness and $p^{\prime}$-convexity properties, along with uniform smoothness and uniform convexity properties, as detailed in Lemma \ref{Con}.
	
	For spaces that are both uniformly smooth and uniformly convex, an equivalent norm exists, transforming $X$ into a $p$-smooth and $p^{\prime}$-convex space for some $1<p\le 2$ and $2\le p^{\prime}<\infty$ (see \cite{MR394135}). Unfortunately, reflexive Banach spaces lack this property, as demonstrated in \cite{MR3446}. However, in the case of super-reflexive Banach spaces, Proposition \ref{Propos} stated in the Appendix asserts that $X$ can be renormed to become uniformly smooth, uniformly convex, and possess $p$-smoothness and $p^{\prime}$-convexity properties for $1<p\le 2$ and $2\le p^{\prime}<\infty$. Notably, Banach spaces that are uniformly smooth and uniformly convex are also super-reflexive (see Remark \ref{Remarkk} in the Appendix).
	
	Consequently, we opt for super-reflexive Banach spaces when investigating controllability results. For $1<p<\infty$, $L^p(\Omega,\mathbb{R})$ and Sobolev spaces $W^{m,p}(\Omega,\mathbb{R})$ with $m>0$, where $\Omega\subset \mathbb{R}^n$ is a measurable set are among few examples of super-reflexive spaces (see \cite{MR2213033}). For detailed definitions of $p$-smooth, $p^{\prime}$-convex Banach spaces and super-reflexive Banach spaces, we refer the reader to see the Appendix section provided at the end of the paper.
\end{remark}
\begin{remark}\label{RM}
	Note that under the assumptions of the paper, we cannot replace Hypothesis (F3) with the Hypothesis of the form:
 \begin{itemize}
\item[(F3$^{\prime}$)]  there exists a function $\eta\in \ L^{\frac{1}{\alpha_1}}(I, \mathbb{R}^+)$ and a constant $d>0$ such that
		\begin{equation}
			\norm{\partial F(t,x)}=\sup\{\norm{z}_{X^*}: z\in \partial F(t,x)\}\le \eta(t)+d\norm{x}, t\in I, x\in X,
		\end{equation}
 \end{itemize}
	This is because, if $q_{\epsilon}, \epsilon>0$ is a mild solution of the problem \eqref{PC1}, then $q_{\epsilon}$ has the form
	\begin{align}\label{MCM1}
		q_{\epsilon}(t)=S_{\alpha}(t)x_0+\int_{0}^{t}(t-s)^{\alpha-1}T_{\alpha}(t-s)Hg_{\epsilon}(s)ds+\int_{0}^{t}(t-s)^{\alpha-1}T_{\alpha}(t-s)Bu_{\epsilon}(s)ds, t\in I.
	\end{align}
	In the above $g_{\epsilon}\in S_{G}(q_{\epsilon})$ and $u_{\epsilon}\in L^2(I,U)$ is given by
	\begin{equation}\label{MCM2}
		u_{\epsilon}(t)=B^*T_{\alpha}^*(a-t)J\left((\epsilon I+R_{\alpha}(a)J)^{-1}\left(z-S_{\alpha}(a)x_0-\int_{0}^{a}(a-s)^{\alpha-1}T_{\alpha}(a-s)Hg_{\epsilon}(s)ds\right)\right).
	\end{equation}
	To prove that $q_{\epsilon}(a)\to z$, we need to show $\{g_{\epsilon}\}_{\epsilon>0}\subset L^1(I, X^*)$ has a weakly convergent subsequence (see STEP-V of the proof of Theorem \ref{MR} given in Section 3).
	Note that $q_{\epsilon}$ has the following bounds (see \textbf{STEP-I} of the proof of Theorem \ref{MR}).
	\begin{align}\label{epsilon}
		\norm{y(t)}\le & M\norm{x_0}+\frac{M}{\Gamma(\alpha)}\norm{H}_{\mathcal{L}(X^*,X)}\Theta+\frac{1}{\epsilon}\frac{M^2\norm{B}^2}{(\Gamma(\alpha))^2} \frac{a^{\alpha}}{\alpha}\left[\norm{z}+M\norm{x_0}+\frac{M}{\Gamma(\alpha)}\norm{H}_{\mathcal{L}(X^*,X)}\Theta\right].
	\end{align}
	Due to the $\frac{1}{\epsilon}$-term in the right hand side of \eqref{epsilon} the sequence of functions $\{g_{\epsilon}\}_{\epsilon>0}$ fails to be bounded or integrably bounded under Hypothesis (F3$^{\prime}$). Therefore, we cannot conclude that $\{g_{\epsilon}\}_{\epsilon>0}\subset L^1(I,X^*)$ has a weakly convergent subsequence.
\end{remark}
\section{Application}
In this section, we provide an example that illustrates the abstract results of this paper. We consider a controlled system modeled by an evolution partial differential equations. The system is described by the classical heat equation involving a multivalued subdifferential term.

Let $\Omega=[0,\pi]$ and $I=[0,a], 0<a<\infty$. We consider the equation
\begin{equation}\label{W1}
	{}^C D_t^\alpha y(t,\theta)=\frac{\partial^2y(t,\theta)}{\partial\theta^2}+ \eta(t,\theta)+\int_{0}^{\pi}L(\theta, \omega)f(t,\omega)d\omega~\text{in}~I\times[0,\pi],
\end{equation}
where $\alpha \in (\frac{1}{2},1)$
\begin{equation}\label{inclusion}
f(t,\theta)\in \partial F(t,\theta, y(t, \theta))~\text{for a.a.}~t\in I, \theta\in [0,\pi].
\end{equation}
Moreover, 
\begin{equation}
	y(t,0)=0, y(t,\pi)=0, ~\text{for}~t\in I, 	
\end{equation}
and
\begin{equation}
	y(0,\theta)=x_0(\theta)~\text{in}~[0,\pi].
\end{equation}
Also $\eta:I\times [0,\pi]\to \mathbb{R}$ is a continuous function defined by
\begin{equation}
	\eta(t,\theta)=\int_{0}^{\pi}K(\theta,\omega)v(t,\omega)d\omega,~~ t\in I, \theta\in [0,\pi].
\end{equation}
Here, $v\in L^2(I\times [0,\pi],\mathbb{R})$ and the kernel $K\in C([0,\pi]\times [0,\pi],\mathbb{R})$ is such that $K(\theta,\xi)=K(\xi,\theta)$ for all $\theta, \xi\in [0,\pi]$. In particular, we can choose $K\in C([0,\pi]\times [0,\pi],\mathbb{R})$ as 
\begin{equation}
	K(t,s)=
	\begin{cases}
		s(\pi-t), ~~\text{if}~0\le s\le t\le \pi\\
		(\pi-s)t, ~~\text{if}~0\le t\le s\le \pi,
	\end{cases}
\end{equation}
or the kernel
\begin{equation}
	K(t,s)=\min\{s,t\}, ~~t,s\in [0,\pi].
\end{equation}
Further, we assume that $L\in C([0,\pi]\times [0,\pi], \mathbb{R})$ and $\partial F$ denotes the Clarke subdifferential of the function $F$ in the third variable. Here, $F=F(t,\theta,r):I\times [0,\pi]\times \mathbb{R}\to \mathbb{R}$ satisfies
\begin{itemize}
	\item[(f1)] $F(\cdot,\cdot,r)$ is measurable on $I\times[0,\pi]$ for all $r\in \mathbb{R}$;
	\item[(f2)] $F(t,\theta,\cdot)$ is Lipschitz for $(t,\theta)\in I\times[0,\pi]$;
	\item[(f3)] there exists a function $\bar{\alpha}\in L^{\frac{1}{\alpha_1}}(I,\mathbb{R}^+)$ with $\alpha_1<\alpha$ such that
	\begin{equation}
		\norm{\partial F(t,\theta,r)}\le \bar{\alpha}(t), ~\text{for all}~r\in \mathbb{R}, ~\text{a.e.}~t\in I,\theta\in [0,\pi];
	\end{equation} 
	\item[(f4)] the function $F(t,\theta,\cdot)$ is regular in the sense of \cite[Definition 3.25]{MR2976197}. 
\end{itemize} 
We underline that due to the lack of convexity of the function $F(t,\theta, \cdot)$, the above problem cannot be formulated as a variational inequality.

The multifunction $\partial F(t,\theta,\cdot):\mathbb{R}\multimap \mathbb{R}$ is generally non-monotone, and it includes the vertical jumps. In a physicist's language, the law is characterized by the generalized gradient of a nonsmooth potential $F$.

We now approach by rewriting the control problem \eqref{W1} as an abstract problem driven by a differential inclusion in the space $X=L^p([0,\pi],\mathbb{R}), p\ge2$ and $U=L^2([0,\pi], \mathbb{R})$. Note that $X$ is a super-reflexive Banach space and $U$ is a separable Hilbert space. To this aim, let us define
\begin{equation*}
	q(t)(\theta)=y(t,\theta), u(t)(\theta)=v(t,\theta), t\in I, \theta\in [0,\pi].
\end{equation*}
We define the operator
$A: D(A)\subset X\to X$ as
\begin{equation}
	Ax(\theta)=x^{\prime\prime}(\theta), \theta\in [0,\pi], x\in X,
\end{equation}
where
\begin{equation}
	D(A)=W^{2,p}([0,\pi], \mathbb{R})\cap W_0^{1,p}([0,\pi], \mathbb{R}).
\end{equation}
Moreover, the spectrum of the operator $A$ is given by $\{-n^2: n\in \mathbb{N}\}$. Then, for every $x\in D(A)$, the operator $A$ can be written as
\begin{equation}
	Ax=\sum_{n=1}^{\infty}-n^2\langle x, w_n\rangle w_n, \langle x,w_n\rangle=\int_{0}^{\pi}x(\theta)w_n(\theta)d\theta,
\end{equation}
where $w_n(\theta)=\sqrt{\frac{2}{\pi}}\sin(n\theta)$ are the normalized eigenfunctions (with respect to the $L^2$ norm) of the operator $A$ corresponding to the eigenvalues $-n^2, n\in \mathbb{N}$. The strongly continuous semigroup $\{T(t)\}_{t\ge 0}$ generated by the operator $A$ is given by
\begin{equation}
	T(t)x=\sum_{n=1}^{\infty}e^{-n^2t}\langle x, w_n\rangle w_n, x\in X.
\end{equation}
Consider the operator $B:U=L^2([0,\pi], \mathbb{R})\to X$ defined as
\begin{equation}
	B(z)(\theta)=\int_{0}^{\pi}K(\theta,\omega)z(\omega)d\omega, z\in U, t\in I, \theta\in [0,\pi].
\end{equation}
Then clearly, the operator $B$ is well-defined and bounded linear. 
We now define $H: X^*\to X$ by
\begin{equation}
	H(x^*)(\theta)=\int_{0}^{\pi}L(\theta,\omega)x^*(\omega)d\omega, x^*\in X^*, \theta\in [0,\pi].
\end{equation}
Consider the function $\Lambda: I\times L^p([0,\pi],\mathbb{R})\to \mathbb{R}$ defined by
\begin{equation}\label{Lambda}
	\Lambda(t,x)=\int_{0}^{\pi}F(t,\theta,x(\theta))d\theta, t\in I, x\in L^p([0,\pi],\mathbb{R}).
\end{equation}
Under the Hypotheses (f1)-(f4), arguing as in the proof of Theorem 3.47 \cite{MR2976197}, we derive the following Lemma.
\begin{lemma}\label{Lemma 5.1}
	The function $\Lambda$ defined in \eqref{Lambda} satisfies the following.
	\begin{itemize}
		\item[($\Lambda$1)] $\Lambda(t,\cdot)$ is well defined and finite on $L^p([0,\pi],\mathbb{R})$ for a.e. $t\in I$.
		\item[($\Lambda$2)] $\Lambda(\cdot, x)$ is measurable on $I$ for all $x\in L^p([0,\pi],\mathbb{R})$.
		\item[($\Lambda$3)] $\Lambda(t,\cdot)$ is Lipschitz on bounded subsets of $L^p([0,\pi],\mathbb{R})$ a.e. $t\in I$.
		\item[($\Lambda$4)] For all $x\in L^p([0,\pi],\mathbb{R}), v\in L^p([0,\pi],\mathbb{R})$ a.e. $t\in I$ we have
		\begin{equation}
			\Lambda^0(t,x;v)= \int_{0}^{\pi}F^0(t,\theta, x(\theta);v(\theta))d\theta.
		\end{equation}
		\item[($\Lambda$5)] For all $x\in L^p([0,\pi],\mathbb{R})$ a.e. $t\in I$ we have
		\begin{equation}
			\partial \Lambda(t,x)= \int_{0}^{\pi}\partial F(t,\theta,x(\theta))d\theta.
		\end{equation}
		\item[($\Lambda$6)] For all $x\in L^p([0,\pi],\mathbb{R})$ a.e. $t\in I$ we have
		\begin{equation}
			\norm{\partial \Lambda(t,x)}\le \gamma(t), 
		\end{equation}
		where $\gamma(t)=\pi^{\frac{1}{p^{\prime}}}\bar{\alpha}(t)$ and $\frac{1}{p}+\frac{1}{p^{\prime}}=1$.
	\end{itemize}
\end{lemma}
By the definition of the generalized Clarke subgradient, we obtain from \eqref{inclusion},
\begin{equation}
	f(t,\theta)v(\theta)\le F^0(t,\theta,q(t)(\theta);v(\theta)), \forall v\in L^p([0,\pi],\mathbb{R}), t\in I.
\end{equation}
Therefore,
\begin{equation}\label{5.17}
	\int_{0}^{\pi} f(t,\theta)v(\theta)d\theta\le \int_{0}^{\pi}F^0(t,\theta,q(t)(\theta);v(\theta))d\theta, t\in I, v\in L^p([0,\pi],\mathbb{R}).
\end{equation}
It is known that the duality pairing between $L^p([0,\pi],\mathbb{R})$ and $L^{p^{\prime}}([0,\pi],\mathbb{R})$ is given by
\begin{equation}
	\langle v,v^*\rangle_{L^p([0,\pi],\mathbb{R}),L^{p^{\prime}}([0,\pi],\mathbb{R})}=\int_{0}^{\pi}v^*(\theta)v(\theta)d\theta, v\in L^p([0,\pi],\mathbb{R}), v^*\in L^{p^{\prime}}([0,\pi],\mathbb{R}). 
\end{equation}
Therefore, we can rewrite equation \eqref{5.17} as
\begin{equation}\label{5.19}
	\langle v,f(t,\cdot)\rangle_{L^p([0,\pi],\mathbb{R}),L^{p^{\prime}}([0,\pi],\mathbb{R})}\le \int_{0}^{\pi}F^0(t,\theta,q(t)(\theta);v(\theta))d\theta, t\in I, v\in L^p([0,\pi],\mathbb{R}).
\end{equation}
Using Lemma \ref{Lemma 5.1} we obtain from \eqref{5.19}
\begin{equation}
	\langle v,f(t,\cdot)\rangle_{L^p([0,\pi],\mathbb{R}),L^{p^{\prime}}([0,\pi],\mathbb{R})}\le \Lambda^0(t,q(t);v), t\in I, v\in L^p([0,\pi],\mathbb{R}).
\end{equation}
The definition of the Clarke generalized subgradient confirms that $f(t,\cdot)\in \partial \Lambda(t,q(t)), t\in I$.

Therefore, the abstract reformulation of equation \eqref{W1} can be given as the following semilinear evolution inclusion in the Banach space $X=L^p([0,\pi],\mathbb{R}), p\ge 2$:
\begin{equation}\label{AF}
	\begin{cases}
		{}^C D_t^\alpha q(t)\in Aq(t)+Bu(t)+H\partial \Lambda(t, q(t)), t\in I\\
		q(0)= x_0.
	\end{cases}
\end{equation}
Of course, the solutions to Problem \eqref{AF} give rise to solutions for \eqref{W1}. 

We now prove the fractional linear control system \eqref{Linear} associated with Problem \eqref{AF} is approximately controllable. In accordance with Theorem \ref{ACOF} we show
\begin{equation}
	B^*T_\alpha^*(a-t)y^*=0~\text{for}~t\in I\Rightarrow x^*=0.
\end{equation}
The adjoint operator $B^*: L^{p^{\prime}}([0,\pi],\mathbb{R})\to L^2([0,\pi], \mathbb{R})$ is given by
\begin{equation}
	B^*x^*(\theta)=\int_{0}^{\pi}K(\theta,\omega)x^*(\omega)d\omega, x^*\in L^{p^{\prime}}([0,\pi],\mathbb{R}), \theta\in [0,\pi].
\end{equation}
We now prove the operator $B^*$ is one-to-one. It is clear that the operator $B^*$ is compact. Therefore, the operator $B^*$ is not invertible, and $0$ is not an eigenvalue of $B^*$. If the map $B^*$ is not one to one then there exists $x^*\in L^{p^{\prime}}([0,\pi],\mathbb{R}), x^*\neq 0$ such that $B^*x^*=0$, which contradict the fact that $0$ is not an eigenvalue of $B^*$. Hence, the map $B^*$ is one-to-one. Therefore, we have
\begin{align}\label{Control condition}
	B^*T_\alpha^*(a-t)y^*=0,~ t\in I\Rightarrow T_\alpha^*(a-t)x^*=0. 
\end{align}
With the help of equation \eqref{Talpha} we obtain
\begin{align}\label{alpha alpha}
	T_\alpha^*(a-t)x^*=\alpha \int_{0}^{\infty}\tau\xi_{\alpha}(\tau)\sum_{n=1}^{\infty}e^{-n^2(a-t)^{\alpha}\tau}\langle x^*, w_n\rangle w_n d\tau=0, ~t\in I.
\end{align}
It follows from expression \eqref{alpha alpha},
\begin{align*}
	\left\langle \sum_{n=1}^{\infty}e^{-n^2(a-t)^{\alpha}\tau}\langle x^*, w_n\rangle w_n, w_n\right\rangle=0, ~\text{for all}~n\in \mathbb{N},
\end{align*}
which implies 
\begin{align*}
	\sum_{n=1}^{\infty}e^{-n^2(a-t)^{\alpha}\tau}\langle x^*, w_n\rangle \left\langle w_n, w_n\right\rangle=0, ~\text{for all}~n\in \mathbb{N}, ~t\in I, ~\tau\in (0,\infty),
\end{align*}
which further implies
\begin{align*}
	\sum_{n=1}^{\infty}e^{-n^2(a-t)^{\alpha}\tau}\langle x^*, w_n\rangle =0, ~\text{for all}~n\in \mathbb{N}, ~t\in I, ~\tau\in (0,\infty).
\end{align*}
Then by \cite[Lemma 3.14]{MR516812} we obtain $\langle x^*,w_n\rangle=0$ for all $n\in \mathbb{N}$ and consequently we have $x^*=0$. This completes the proof.

Thus we verify that Hypothesis (T) and (U) holds. Also, by Lemma \ref{Lemma 5.1} the function $\Lambda$ satisfies Hypotheses ($\Lambda$1)-($\Lambda$6). Note that based on Hypotheses ($\Lambda$3) and ($\Lambda$6), by virtue of Lebourg Mean Value Theorem \cite[Theorem 2.3.7]{MR1058436} the map $\Lambda(t,\cdot):X\to\mathbb{R}$ is Lipschitz continuous. Therefore, the function $F(t,x)=\Lambda(t,x)$ satisfies all the conditions of Theorem \ref{MR}. Hence, by Theorem \ref{MR}, the control system \eqref{W1} is approximately controllable in $I$. 
\section{Appendix}
We start by introducing uniformly convex and uniformly smooth spaces.
\begin{definition}\label{uniformly convex}
	A normed space $X$ is said to be uniformly convex if for any $\epsilon\in (0,2]$ there exists a $\delta>0$ such that if $x,y\in X$ with $\norm{x}\le 1, \norm{y}\le 1$ and $\norm{x-y}\ge \epsilon$ then,
	\begin{equation}
		\norm{\frac{x+y}{2}}\le 1-\delta.
	\end{equation}
\end{definition}
\begin{definition}\label{uniformly smooth}
	A Banach space $X$ is called uniformly smooth if for any $\epsilon>0$ there exists $\delta>0$ such that for all $x,y\in X$ with $\norm{x}=1, \norm{y}\le \delta$ the inequality
	\begin{equation}
		\frac{1}{2}\left[\norm{x+y}+\norm{x-y}\right]-1\le \epsilon \norm{y},
	\end{equation} 
	holds.
\end{definition}
The simplest example of a space that is both uniformly convex and uniformly smooth is the Hilbert space.

We immediately remark that there is a duality relationship between uniform smoothness and uniform convexity.
\begin{theorem}\label{duality relationship}
	Let $X$ be a Banach space. Then, the following statements are equivalent.
	\begin{itemize}
		\item[(i)] $X$ is uniformly smooth;
		\item[(ii)] $X^*$ is uniformly convex.
	\end{itemize}
\end{theorem}
In the realm of Banach spaces, two crucial measures are the modulus of convexity $\delta_X$ and the modulus of smoothness $\rho_X$.	\begin{definition}\cite{MR394135}
	If $X$ is a Banach space, its modulus of convexity $\delta_X$ is defined as follows
	\begin{equation}
		\forall \epsilon\in [0,2], \delta_X(\epsilon)=\inf\left\{1-\norm{\frac{x+y}{2}}: \norm{x}= 1, \norm{y}= 1, \norm{x-y}= \epsilon\right\}.
	\end{equation} 
\end{definition}
\begin{definition}\cite{MR394135}
	If $X$ is a Banach space, its modulus of smoothness $\rho_X$ is defined by
	\begin{equation}
		\forall t\in [0,\infty), \rho_X(t)=\sup\left\{\frac{\norm{x+ty}+\norm{x-ty}}{2}-1:\norm{x}=\norm{y}=1\right\}.
	\end{equation} 
\end{definition}
\begin{definition}\cite{MR394135}
	We shall say that a Banach space $X$ is $p$-smooth, $1<p\le 2$ if there exists an equivalent norm on $X$ for which the modulus of smoothness $\rho_X$ satisfies 
	\begin{equation}\label{mos}
		\forall t>0, \rho_X(t)\le Ct^p,
	\end{equation}
	for some constant $C>0$.
\end{definition}
\begin{definition}\cite{MR394135}
	We shall say that a Banach space $X$ is $p^{\prime}$-convex, $2\le p^{\prime}<\infty$ if there exists an equivalent norm on $X$ for which the modulus of convexity $\delta_X$ satisfies 
	\begin{equation}\label{moc}
		\forall \epsilon>0,	\delta_X(\epsilon)\ge  C\epsilon^{p^{\prime}},
	\end{equation}
	for some constant $C>0$.
\end{definition}
Various insights arise from these definitions, including relationships between smoothness and convexity, conditions for uniform smoothness and uniform convexity, and the reflexive nature of certain spaces.
\begin{remark}
	\begin{itemize}
		\item[(i)] If the modulus of smoothness $\rho_X$ satisfies \eqref{mos}, then from Dvoretzky's Theorem \cite{dvoretzky1964some}, we necessarily have $1<p\le 2$.
		\item[(2)] If the modulus of convexity $\delta_X$ satisfies \eqref{moc}, then it follows from Dvoretzky's Theorem \cite{dvoretzky1964some} that necessarily $p^{\prime}\ge 2$. 
	\end{itemize}
\end{remark}
We now introduce super-reflexive spaces.
\begin{definition}\cite{MR394135}
	A Banach space $X$ is super-reflxeive if and only if there exists an integer $n$ and an $\epsilon>0$ such that for every $n$-tuple $(x_1, x_2, \cdot \cdot \cdot, x_n)$ in the unit ball of $X$:
	\begin{equation}
		\inf_{1\le k\le n}\norm{\sum_{1\le i\le k}x_i-\sum_{k<i\le n}x_i}\le n(1-\epsilon).
	\end{equation}
\end{definition}
\begin{lemma}
	If a Banach space $X$ is such that $\delta_X(2-\epsilon)>0$ for some $\epsilon>0$, then $X$ is super-reflexive.
\end{lemma}
\begin{remark}\cite{MR394135}\label{Remarkk}
	\begin{itemize}
		\item[(1)] If $X^*$ is $p$-smooth, then $X$ is $p^{\prime}$-convex with $\frac{1}{p}+\frac{1}{p^{\prime}}=1$.
		\item[(2)] Every uniformly convex or uniformly smooth Banach space is $p^{\prime}$-convex and $p$-smooth for some $p^{\prime}<\infty$ and $p>1$. 
		\item[(3)] A normed space $X$ is uniformly smooth if and only if 
		\begin{equation}
			\lim_{t\to 0^+}\frac{\rho_X(t)}{t}=0.
		\end{equation}
		\item[(4)] A Banach space $X$ is uniformly convex if and only if $\delta_X(\epsilon)>0$ for all $\epsilon\in (0,2]$.
		\item[(5)] Every uniformly convex or uniformly smooth Banach space is reflexive.
		\item[(6)] Every uniformly smooth or uniformly convex space is super-reflexive.
	\end{itemize}
\end{remark}
\begin{ex}
	Noteworthy examples of super-reflexive spaces include
	the spaces $L^p, 1<p<\infty$ and $W^{m,p}, m\in \mathbb{N}\cup\{0\}, 1<p<\infty$. 
\end{ex}
Furthermore, the following proposition asserts that any super-reflexive Banach space can be equipped with an equivalent norm that is both uniformly convex and uniformly smooth, featuring moduli of convexity and smoothness of power type. 
\begin{proposition}\label{Propos}\cite[Proposition 5.2]{MR1211634}
	Let $X$ be a super-reflexive Banach space. Then, an equivalent norm exists on $X$, which is both uniformly convex and uniformly smooth, with moduli of convexity and smoothness in power type.
\end{proposition}

\bibliographystyle{unsrtnat}
\bibliography{references}  

\begin{thebibliography}{48}
\providecommand{\natexlab}[1]{#1}
\providecommand{\url}[1]{\texttt{#1}}
\expandafter\ifx\csname urlstyle\endcsname\relax
  \providecommand{\doi}[1]{doi: #1}\else
  \providecommand{\doi}{doi: \begingroup \urlstyle{rm}\Url}\fi

\bibitem[Pang and Stewart(2008)]{MR2375486}
Jong-Shi Pang and David~E. Stewart.
\newblock Differential variational inequalities.
\newblock \emph{Math. Program.}, 113\penalty0 (2, Ser. A):\penalty0 345--424,
  2008.
\newblock ISSN 0025-5610.

\bibitem[Panagiotopoulos(1985)]{MR896909}
P.~D. Panagiotopoulos.
\newblock \emph{Inequality problems in mechanics and applications}.
\newblock Birkh\"{a}user Boston, Inc., Boston, MA, 1985.
\newblock ISBN 0-8176-3094-5.
\newblock Convex and nonconvex energy functions.

\bibitem[Panagiotopoulos(1993)]{MR1385670}
P.~D. Panagiotopoulos.
\newblock \emph{Hemivariational inequalities}.
\newblock Springer-Verlag, Berlin, 1993.
\newblock ISBN 3-540-54963-3.
\newblock Applications in mechanics and engineering.

\bibitem[Liu et~al.(2018)Liu, Zeng, and Motreanu]{MR3871422}
Zhenhai Liu, Shengda Zeng, and Dumitru Motreanu.
\newblock Partial differential hemivariational inequalities.
\newblock \emph{Adv. Nonlinear Anal.}, 7\penalty0 (4):\penalty0 571--586, 2018.
\newblock ISSN 2191-9496.

\bibitem[Naniewicz and Panagiotopoulos(1995)]{MR1304257}
Z.~Naniewicz and P.~D. Panagiotopoulos.
\newblock \emph{Mathematical theory of hemivariational inequalities and
  applications}, volume 188 of \emph{Monographs and Textbooks in Pure and
  Applied Mathematics}.
\newblock Marcel Dekker, Inc., New York, 1995.
\newblock ISBN 0-8247-9330-7.

\bibitem[Liu and Papageorgiou(2022)]{MR4375652}
Zhenhai Liu and Nikolaos~S. Papageorgiou.
\newblock Double phase {D}irichlet problems with unilateral constraints.
\newblock \emph{J. Differential Equations}, 316:\penalty0 249--269, 2022.
\newblock ISSN 0022-0396.

\bibitem[Abbas et~al.(2012)Abbas, Benchohra, and
  N'Gu\'{e}r\'{e}kata]{MR2962045}
Sa\"{\i}d Abbas, Mouffak Benchohra, and Gaston~M. N'Gu\'{e}r\'{e}kata.
\newblock \emph{Topics in fractional differential equations}, volume~27 of
  \emph{Developments in Mathematics}.
\newblock Springer, New York, 2012.
\newblock ISBN 978-1-4614-4035-2.

\bibitem[Koeller(1984)]{MR747787}
R.~C. Koeller.
\newblock Applications of fractional calculus to the theory of viscoelasticity.
\newblock \emph{Trans. ASME J. Appl. Mech.}, 51\penalty0 (2):\penalty0
  299--307, 1984.
\newblock ISSN 0021-8936.

\bibitem[Fabrizio(2014)]{MR3146657}
Mauro Fabrizio.
\newblock Fractional rheological models for thermomechanical systems.
  {D}issipation and free energies.
\newblock \emph{Fract. Calc. Appl. Anal.}, 17\penalty0 (1):\penalty0 206--223,
  2014.
\newblock ISSN 1311-0454.

\bibitem[Hilfer(2000)]{MR1890104}
R.~Hilfer, editor.
\newblock \emph{Applications of fractional calculus in physics}.
\newblock World Scientific Publishing Co., Inc., River Edge, NJ, 2000.
\newblock ISBN 981-02-3457-0.

\bibitem[Curtain and Pritchard(1978)]{MR516812}
Ruth~F. Curtain and Anthony~J. Pritchard.
\newblock \emph{Infinite dimensional linear systems theory}, volume~8 of
  \emph{Lecture Notes in Control and Information Sciences}.
\newblock Springer-Verlag, Berlin-New York, 1978.
\newblock ISBN 3-540-08961-6.

\bibitem[Ainseba(2002)]{MR1943766}
Bedr'Eddine Ainseba.
\newblock Exact and approximate controllability of the age and space population
  dynamics structured model.
\newblock \emph{J. Math. Anal. Appl.}, 275\penalty0 (2):\penalty0 562--574,
  2002.
\newblock ISSN 0022-247X.

\bibitem[Singh and Shukla(2023)]{MR4577651}
Ajeet Singh and Anurag Shukla.
\newblock Approximate controllability of the semilinear population dynamics
  system with diffusion.
\newblock \emph{Math. Methods Appl. Sci.}, 46\penalty0 (7):\penalty0
  8418--8429, 2023.
\newblock ISSN 0170-4214.

\bibitem[Ani\c{t}a(2000)]{MR1797596}
Sebastian Ani\c{t}a.
\newblock \emph{Analysis and control of age-dependent population dynamics},
  volume~11 of \emph{Mathematical Modelling: Theory and Applications}.
\newblock Kluwer Academic Publishers, Dordrecht, 2000.
\newblock ISBN 0-7923-6639-5.

\bibitem[Clarke(1990)]{MR1058436}
F.~H. Clarke.
\newblock \emph{Optimization and nonsmooth analysis}, volume~5 of
  \emph{Classics in Applied Mathematics}.
\newblock Society for Industrial and Applied Mathematics (SIAM), Philadelphia,
  PA, second edition, 1990.
\newblock ISBN 0-89871-256-4.

\bibitem[Metzler and Klafter(2000)]{MR1809268}
Ralf Metzler and Joseph Klafter.
\newblock The random walk's guide to anomalous diffusion: a fractional dynamics
  approach.
\newblock \emph{Phys. Rep.}, 339\penalty0 (1):\penalty0 77, 2000.
\newblock ISSN 0370-1573.

\bibitem[Meerschaert and Scalas(2006)]{MR2263769}
Mark~M. Meerschaert and Enrico Scalas.
\newblock Coupled continuous time random walks in finance.
\newblock \emph{Phys. A}, 370\penalty0 (1):\penalty0 114--118, 2006.
\newblock ISSN 0378-4371.

\bibitem[Benson et~al.(2000)Benson, Wheatcraft, and
  Meerschaert]{benson2000application}
David~A Benson, Stephen~W Wheatcraft, and Mark~M Meerschaert.
\newblock Application of a fractional advection-dispersion equation.
\newblock \emph{Water resources research}, 36\penalty0 (6):\penalty0
  1403--1412, 2000.

\bibitem[Burdzy(1993)]{MR1278077}
Krzysztof Burdzy.
\newblock Some path properties of iterated {B}rownian motion.
\newblock In \emph{Seminar on {S}tochastic {P}rocesses, 1992 ({S}eattle, {WA},
  1992)}, volume~33 of \emph{Progr. Probab.}, pages 67--87. Birkh\"{a}user
  Boston, Boston, MA, 1993.

\bibitem[Concezzi and Spigler(2015)]{MR3316530}
Moreno Concezzi and Renato Spigler.
\newblock Some analytical and numerical properties of the {M}ittag-{L}effler
  functions.
\newblock \emph{Fract. Calc. Appl. Anal.}, 18\penalty0 (1):\penalty0 64--94,
  2015.
\newblock ISSN 1311-0454.

\bibitem[Diethelm and Ford(2002)]{diethelm2002analysis}
Kai Diethelm and Neville~J Ford.
\newblock Analysis of fractional differential equations.
\newblock \emph{Journal of Mathematical Analysis and Applications},
  265\penalty0 (2):\penalty0 229--248, 2002.

\bibitem[Baeumer et~al.(2009)Baeumer, Meerschaert, and Nane]{MR2491905}
Boris Baeumer, Mark~M. Meerschaert, and Erkan Nane.
\newblock Brownian subordinators and fractional {C}auchy problems.
\newblock \emph{Trans. Amer. Math. Soc.}, 361\penalty0 (7):\penalty0
  3915--3930, 2009.
\newblock ISSN 0002-9947.

\bibitem[Liu and Li(2015)]{MR3280853}
Zhenhai Liu and Xiuwen Li.
\newblock Approximate controllability for a class of hemivariational
  inequalities.
\newblock \emph{Nonlinear Anal. Real World Appl.}, 22:\penalty0 581--591, 2015.
\newblock ISSN 1468-1218.

\bibitem[Liu et~al.(2015)Liu, Li, and Motreanu]{MR3411720}
Zhenhai Liu, Xiuwen Li, and Dumitru Motreanu.
\newblock Approximate controllability for nonlinear evolution hemivariational
  inequalities in {H}ilbert spaces.
\newblock \emph{SIAM J. Control Optim.}, 53\penalty0 (5):\penalty0 3228--3244,
  2015.
\newblock ISSN 0363-0129.

\bibitem[Li et~al.(2016)Li, Li, and Liu]{MR3512753}
Yunxiang Li, Xiuwen Li, and Yiliang Liu.
\newblock On the approximate controllability for fractional evolution
  hemivariational inequalities.
\newblock \emph{Math. Methods Appl. Sci.}, 39\penalty0 (11):\penalty0
  3088--3101, 2016.
\newblock ISSN 0170-4214.

\bibitem[Liu et~al.(2019)Liu, Wang, and O'Regan]{MR3927856}
Xianghu Liu, JinRong Wang, and D.~O'Regan.
\newblock On the approximate controllability for fractional evolution
  inclusions of {S}obolev and {C}larke subdifferential type.
\newblock \emph{IMA J. Math. Control Inform.}, 36\penalty0 (1):\penalty0 1--17,
  2019.
\newblock ISSN 0265-0754.

\bibitem[Chang and Liu(2020)]{MR4128438}
Yong-Kui Chang and Xiaojing Liu.
\newblock Time-varying integro-differential inclusions with {C}larke
  sub-differential and non-local initial conditions: existence and approximate
  controllability.
\newblock \emph{Evol. Equ. Control Theory}, 9\penalty0 (3):\penalty0 845--863,
  2020.
\newblock ISSN 2163-2472.

\bibitem[Zhao et~al.(2023)Zhao, Liu, and Liu]{MR4560898}
Jing Zhao, Zhenhai Liu, and Yongjian Liu.
\newblock Approximate controllability of non-autonomous second-order evolution
  hemivariational inequalities with nonlocal conditions.
\newblock \emph{Appl. Anal.}, 102\penalty0 (1):\penalty0 23--37, 2023.
\newblock ISSN 0003-6811.

\bibitem[Wang et~al.(2019)Wang, Liu, and O'Regan]{MR3937061}
JinRong Wang, Xianghu Liu, and D.~O'Regan.
\newblock On the approximate controllability for {H}ilfer fractional evolution
  hemivariational inequalities.
\newblock \emph{Numer. Funct. Anal. Optim.}, 40\penalty0 (7):\penalty0
  743--762, 2019.
\newblock ISSN 0163-0563.

\bibitem[Kamenskii et~al.(2001)Kamenskii, Obukhovskii, and Zecca]{MR1831201}
Mikhail Kamenskii, Valeri Obukhovskii, and Pietro Zecca.
\newblock \emph{Condensing multivalued maps and semilinear differential
  inclusions in {B}anach spaces}, volume~7 of \emph{De Gruyter Series in
  Nonlinear Analysis and Applications}.
\newblock Walter de Gruyter \& Co., Berlin, 2001.
\newblock ISBN 3-11-016989-4.

\bibitem[Mig\'{o}rski et~al.(2013)Mig\'{o}rski, Ochal, and Sofonea]{MR2976197}
Stanis\'{l}aw Mig\'{o}rski, Anna Ochal, and Mircea Sofonea.
\newblock \emph{Nonlinear inclusions and hemivariational inequalities},
  volume~26 of \emph{Advances in Mechanics and Mathematics}.
\newblock Springer, New York, 2013.
\newblock ISBN 978-1-4614-4231-8; 978-1-4614-4232-5.
\newblock Models and analysis of contact problems.

\bibitem[Gasi\'{n}ski and Papageorgiou(2016)]{MR3524637}
Leszek Gasi\'{n}ski and Nikolaos~S. Papageorgiou.
\newblock \emph{Exercises in analysis. {P}art 2. {N}onlinear analysis}.
\newblock Problem Books in Mathematics. Springer, Cham, 2016.
\newblock ISBN 978-3-319-27815-5; 978-3-319-27817-9.

\bibitem[Suechoei and Sa~Ngiamsunthorn(2020)]{MR4076740}
Apassara Suechoei and Parinya Sa~Ngiamsunthorn.
\newblock Existence uniqueness and stability of mild solutions for semilinear
  {$\psi$}-{C}aputo fractional evolution equations.
\newblock \emph{Adv. Difference Equ.}, pages Paper No. 114, 28, 2020.
\newblock ISSN 1687-1839.

\bibitem[Glicksberg(1952)]{MR46638}
I.~L. Glicksberg.
\newblock A further generalization of the {K}akutani fixed theorem, with
  application to {N}ash equilibrium points.
\newblock \emph{Proc. Amer. Math. Soc.}, 3:\penalty0 170--174, 1952.
\newblock ISSN 0002-9939.

\bibitem[El-Borai(2002)]{MR1903295}
Mahmoud~M. El-Borai.
\newblock Some probability densities and fundamental solutions of fractional
  evolution equations.
\newblock \emph{Chaos Solitons Fractals}, 14\penalty0 (3):\penalty0 433--440,
  2002.
\newblock ISSN 0960-0779.

\bibitem[Zhou and Jiao(2010)]{MR2579471}
Yong Zhou and Feng Jiao.
\newblock Existence of mild solutions for fractional neutral evolution
  equations.
\newblock \emph{Comput. Math. Appl.}, 59\penalty0 (3):\penalty0 1063--1077,
  2010.
\newblock ISSN 0898-1221.

\bibitem[Zeidler(1990)]{MR1033498}
Eberhard Zeidler.
\newblock \emph{Nonlinear functional analysis and its applications. {II}/{B}}.
\newblock Springer-Verlag, New York, 1990.
\newblock ISBN 0-387-97167-X.
\newblock Nonlinear monotone operators, Translated from the German by the
  author and Leo F. Boron.

\bibitem[Chidume(2009)]{MR2504478}
Charles Chidume.
\newblock \emph{Geometric properties of {B}anach spaces and nonlinear
  iterations}, volume 1965 of \emph{Lecture Notes in Mathematics}.
\newblock Springer-Verlag London, Ltd., London, 2009.
\newblock ISBN 978-1-84882-189-7.

\bibitem[Mahmudov(2003)]{MR2046377}
Nazim~I. Mahmudov.
\newblock Approximate controllability of semilinear deterministic and
  stochastic evolution equations in abstract spaces.
\newblock \emph{SIAM J. Control Optim.}, 42\penalty0 (5):\penalty0 1604--1622,
  2003.
\newblock ISSN 0363-0129.

\bibitem[Asplund(1967)]{MR222610}
Edgar Asplund.
\newblock Averaged norms.
\newblock \emph{Israel J. Math.}, 5:\penalty0 227--233, 1967.
\newblock ISSN 0021-2172.

\bibitem[Pinaud and Henr\'{\i}quez(2020)]{MR4104454}
Matthieu~F. Pinaud and Hern\'{a}n~R. Henr\'{\i}quez.
\newblock Controllability of systems with a general nonlocal condition.
\newblock \emph{J. Differential Equations}, 269\penalty0 (6):\penalty0
  4609--4642, 2020.
\newblock ISSN 0022-0396.

\bibitem[Arora et~al.(2022)Arora, Mohan, and Dabas]{MR4336468}
S.~Arora, Manil~T. Mohan, and J.~Dabas.
\newblock Existence and approximate controllability of non-autonomous
  functional impulsive evolution inclusions in {B}anach spaces.
\newblock \emph{J. Differential Equations}, 307:\penalty0 83--113, 2022.
\newblock ISSN 0022-0396.

\bibitem[Bartle(1995)]{bartle2014elements}
Robert~G. Bartle.
\newblock The elements of integration and {L}ebesgue measure.
\newblock Wiley Classics Library, pages xii+179. John Wiley \& Sons, Inc., New
  York, 1995.
\newblock ISBN 0-471-04222-6.
\newblock Containing a corrected reprint of the 1966 original [{{\i}t The
  elements of integration}, Wiley, New York; MR0200398 (34 \#293)], A
  Wiley-Interscience Publication.

\bibitem[Pisier(1975)]{MR394135}
Gilles Pisier.
\newblock Martingales with values in uniformly convex spaces.
\newblock \emph{Israel J. Math.}, 20\penalty0 (3-4):\penalty0 326--350, 1975.
\newblock ISSN 0021-2172.

\bibitem[Day(1941)]{MR3446}
Mahlon~M. Day.
\newblock Reflexive {B}anach spaces not isomorphic to uniformly convex spaces.
\newblock \emph{Bull. Amer. Math. Soc.}, 47:\penalty0 313--317, 1941.
\newblock ISSN 0002-9904.

\bibitem[Alber and Ryazantseva(2006)]{MR2213033}
Yakov Alber and Irina Ryazantseva.
\newblock \emph{Nonlinear ill-posed problems of monotone type}.
\newblock Springer, Dordrecht, 2006.
\newblock ISBN 978-1-4020-4395-6; 1-4020-4395-3.

\bibitem[Dvoretzky(1964)]{dvoretzky1964some}
Aryeh Dvoretzky.
\newblock Some results on convex bodies and banach spaces.
\newblock \emph{Matematika}, 8\penalty0 (1):\penalty0 73--102, 1964.

\bibitem[Deville et~al.(1993)Deville, Godefroy, and Zizler]{MR1211634}
Robert Deville, Gilles Godefroy, and V\'{a}clav Zizler.
\newblock \emph{Smoothness and renormings in {B}anach spaces}, volume~64 of
  \emph{Pitman Monographs and Surveys in Pure and Applied Mathematics}.
\newblock Longman Scientific \& Technical, Harlow; copublished in the United
  States with John Wiley \& Sons, Inc., New York, 1993.
\newblock ISBN 0-582-07250-6.

\end{thebibliography}






\end{document}